\documentclass[12pt]{amsart}
\usepackage{amssymb,amsfonts,amsmath,amsopn,amstext,amscd,color,
latexsym,mathrsfs,stackrel,url,verbatim}
\usepackage[all]{xy}
\usepackage{mathabx}
\usepackage{hyperref}
\usepackage[utf8]{inputenc}

\input xy
\xyoption{all}

\usepackage[active]{srcltx}

\theoremstyle{remark}

\newtheorem{para}{\bf}[section]

\theoremstyle{definition}

\newtheorem{dfn}[para]{Definition}

\theoremstyle{plain}

\newtheorem{thm}[para]{Theorem}

\newtheorem{lemma}[para]{Lemma}
\newtheorem{cor}[para]{Corollary}
\newtheorem{prop}[para]{Proposition}

\newcommand{\vpi}{{\varpi}}

\newcommand{\Ga}{\Gamma}

\newcommand{\bbA}{{\mathbb A}}

\newcommand{\bbD}{{\mathbb D}}

\newcommand{\bbF}{{\mathbb F}}
\newcommand{\bbG}{{\mathbb G}}

\newcommand{\bbN}{{\mathbb N}}

\newcommand{\bbP}{{\mathbb P}}
\newcommand{\bbQ}{{\mathbb Q}}
\newcommand{\bbR}{{\mathbb R}}

\newcommand{\bbU}{{\mathbb U}}

\newcommand{\bbX}{{\mathbb X}}
\newcommand{\bbY}{{\mathbb Y}}
\newcommand{\bbZ}{{\mathbb Z}}

\newcommand{\bL}{{\bf L}}

\newcommand{\frg}{{\mathfrak g}}

\newcommand{\frn}{{\mathfrak n}}
\newcommand{\fro}{{\mathfrak o}}

\newcommand{\frt}{{\mathfrak t}}

\newcommand{\frU}{{\mathfrak U}}

\newcommand{\frX}{{\mathfrak X}}
\newcommand{\frY}{{\mathfrak Y}}

\newcommand{\cA}{{\mathcal A}}
\newcommand{\cB}{{\mathcal B}}

\newcommand{\cE}{{\mathcal E}}
\newcommand{\cF}{{\mathcal F}}
\newcommand{\cG}{{\mathcal G}}
\newcommand{\cH}{{\mathcal H}}
\newcommand{\cI}{{\mathcal I}}

\newcommand{\cL}{{\mathcal L}}
\newcommand{\cM}{{\mathcal M}}

\newcommand{\cO}{{\mathcal O}}
\newcommand{\cP}{{\mathcal P}}
\newcommand{\cQ}{{\mathcal Q}}

\newcommand{\cT}{{\mathcal T}}
\newcommand{\cU}{{\mathcal U}}

\newcommand{\cX}{{\mathcal X}}
\newcommand{\cY}{{\mathcal Y}}

\newcommand{\sD}{{\mathscr D}}

\newcommand{\Q}{{\mathbb Q}}

\newcommand{\ot}{\otimes}

\newcommand{\ocE}{\overline{\cE}}
\newcommand{\ocO}{\overline{\cO}}
\newcommand{\ocF}{\overline{\cF}}

\newcommand{\Hom}{{\rm Hom}}

\newcommand{\Lie}{{\rm{Lie}}}

\newcommand{\Spec}{{\rm Spec\,}}

\newcommand{\sta}{\stackrel}

\newcommand{\hra}{\hookrightarrow}
\newcommand{\hla}{\hookleftarrow}
\newcommand{\im}{{\rm im}}
\newcommand{\midc}{{\; | \;}}
\newcommand{\ord}{{\rm ord}}
\newcommand{\ra}{\rightarrow}
\newcommand{\la}{\leftarrow}

\newcommand{\osD}{{\overline{\sD}}} 
\newcommand{\hsD}{{\widehat{\sD}}} 

\newcommand{\hf}{{\hat{f}}}
\newcommand{\hg}{{\hat{g}}}

\newcommand{\hi}{{\hat{\imath}}}

\newcommand{\bul}{\bullet}

\newcommand{\car}{\stackrel{\simeq}{\longrightarrow}}

\newcommand{\varep}{\varepsilon}

\newcommand{\Lam}{\Lambda}
\newcommand{\rig}{\rightarrow}
\newcommand{\trig}{\twoheadrightarrow}
\newcommand{\thra}{\twoheadrightarrow}
\newcommand{\hrig}{\hookrightarrow}
\newcommand{\GG}{{\mathcal G}}\newcommand{\TT}{{\mathcal T}}\newcommand{\NN}{{\mathcal N}}

\newcommand{\uk}{\underline{k}}

\newcommand{\ut}{\underline{t}}

\newcommand{\uder}{\underline{\partial}}

\newcommand{\unu}{\underline{\nu}}

\begin{document}
\title{Intermediate extensions and crystalline distribution algebras}
\author{Christine Huyghe}
\address{IRMA, Universit\'e de Strasbourg, 7 rue Ren\'e Descartes, 67084 Strasbourg cedex, France}
\email{huyghe@math.unistra.fr}
\author{Tobias Schmidt}
\address{IRMAR, Universit\'e de Rennes 1, Campus Beaulieu, 35042 Rennes cedex, France}
\email{Tobias.Schmidt@univ-rennes1.fr}

\thanks{}

\begin{abstract} Let $G$ be a connected split reductive group over a complete discrete valuation ring of mixed characteristic.
We use the theory of intermediate extensions due to Abe-Caro and arithmetic Beilinson-Bernstein localization to classify irreducible modules over the crystalline distribution algebra of $G$ in terms of overconvergent isocrystals on locally closed subspaces in the (formal) flag variety of $G$. We treat the case of ${\rm SL}_2$ as an example. 
\end{abstract}

\maketitle

\tableofcontents

\newpage 

\section{Introduction}
Let $\fro$ denote a complete discrete valuation ring of mixed characteristic $(0,p)$, with fraction field $K$ and perfect residue field $k$.
Let $G$ be a connected split reductive group over $\fro$ with $K$-Lie algebra $\frg=Lie(G)\otimes\Q$.

\vskip5pt

 In \cite{HS1} we have introduced and studied the crystalline distribution algebra $D^\dagger(\GG)$ associated to the $p$-adic completion $\GG$ of $G$. It is a certain weak completion of the classical universal enveloping algebra $U(\frg)$.
The interest in the algebra $D^\dagger(\GG)$ comes at least from two sources. On the one hand, it has the universal property to act as global arithmetic differential operators (in the sense of Berthelot \cite{BerthelotDI}) on any formal $\fro$-scheme which has a $\GG$-action. On the other hand, $D^\dagger(\GG)$ is canonically isomorphic to Emerton's analytic distribution algebra $D^{an}(G^\circ)$ as introduced in \cite{EmertonA}. Here, $G^{\circ}$ equals the rigid-analytic generic fibre of the formal completion of $G$ along its unit section. Analytic distribution algebras can be useful tools in the locally analytic representation theory of $p$-adic Lie groups, as we briefly recall. Let $G(K)$ be the group of $K$-valued points of $G$ and let $G(n)^{\circ}$ be the $n$-th rigid-analytic congruence subgroup of $G$ 
(with $G(0)^{\circ}=G^{\circ}$). Any irreducible admissible locally analytic $G(K)$-representation $V$ has an infinitesimal character and a level. The latter equals the least natural number $n\geq 0$ such that $V_{G(n)^{\circ}-{\rm an}}\neq 0$, i.e. such that $V$ contains a nonzero $G(n)^{\circ}$-analytic vector. 
The dual space $(V_{G(n)^{\circ}-{\rm an}})'$ is naturally a module over the ring $D^{an}(G(n)^\circ)$. Let ${\rm Rep}^n_{\theta}(G(K))$ be the full subcategory of admissible representations $V$ of character $\theta$, which are generated by their $G(n)^{\circ}$-analytic vectors. Let ${\rm Coh}(D^{an}(G(n)^\circ)_{\theta})$ be the category of coherent modules over the central reduction $D^{an}(G(n)^\circ)_{\theta}$. The formation $V\mapsto (V_{G(n)^{\circ}-{\rm an}})'$ is a faithful and exact functor 
$${\rm Rep}_{\theta}^n(G(K))\longrightarrow {\rm Coh}(D^{an}(G(n)^\circ)_{\theta}),\hskip5pt $$
which detects irreducibility: if the module $(V_{G(n)^{\circ}-{\rm an}})'$ is irreducible, then $V$ is an 
irreducible object in ${\rm Rep}^n_{\theta}(G(K))$. 

\vskip5pt

In this article, we only consider the simplest case: representations of level zero and with trivial infinitesimal character $\theta_0$. We then propose to study the irreducible modules over the ring $D^\dagger(\GG)_{\theta_0}$. Our approach will be geometric through some crystalline version of localization, 
similar to the classical procedure of localizing $U(\frg)$-modules. Recall that, in the classical setting of $U(\frg)$-modules, a combination of the Beilinson-Bernstein localization theorem over the flag variety of $\frg$ together with the formalism of intermediate extensions \cite{BB81,BK81,Hotta} produces a geometric classification of many irreducible modules, namely those 
 which localize to $D$-modules which are {\it holonomic}. 
 
 \vskip5pt
 
 Let in the following $B\subset G$ be a Borel subgroup scheme. In \cite{HS2} we have established an analogue of the Beilinson-Bernstein theorem for arithmetic differential operators on the formal completion $\cP$ of the flag scheme $P=G/B$ of $G$: one has a canonical isomorphism $H^0(\cP, \sD_{\cP}^{\dagger})\simeq D^\dagger(\GG)_{\theta_0}$ and the global sections functor furnishes an equivalence between the category of coherent arithmetic $\sD_{\cP}^{\dagger}$-modules and coherent $D^\dagger(\GG)_{\theta_0}$-modules respectively. An explicit quasi-inverse is given by the adjoint functor 
 ${\mathscr Loc}(M)=  \sD_{\cP}^{\dagger} \otimes_ {D^\dagger(\GG)} M$. This allows to pass back and forth between modules over  $D^\dagger(\GG)_{\theta_0}$ and sheaves on $\cP$.
 
 \vskip5pt 
 
 On the other hand, Abe-Caro have recently developed a theory of weights in $p$-adic cohomology \cite{AbeCaro} building on the six functor formalism 
 for Caro's overholonomic complexes \cite{Caro_ovhol}. On the way, they also defined an intermediate extension functor for arithmetic $\sD$-modules and investigated some of its properties. We then use a combination of Abe-Caro's theory, specialized to the flag variety, 
 and localization to obtain classification results for irreducible $D^\dagger(\GG)$-modules, in analogy to the classical setting of $U(\frg)$-modules.
 
 \vskip5pt 
 
 Our main result is the following: we call a nonzero $D^\dagger(\GG)_{\theta_0}$-module $M$ {\it geometrically $F$-holonomic}, if its
 localization ${\mathscr Loc}(M)$ has a Frobenius structure and is holonomic. We then consider the parameter set of pairs $(Y,\cE)$ where $Y\subset \cP_s$ 
is a connected smooth locally closed subvariety of the special fibre $\cP_s$ , $X$ its Zariski closure, and $\cE$ is an irreducible overconvergent $F$-isocrystal on the couple $\bbY=(Y,X)$. Two pairs $(Y,\cE)$ and  $(Y',\cE')$ are equivalent if $X=X'$ and the two isocrystals $\cE, \cE'$ coincide on an open dense subset of $X$. Given such a pair $(Y,\cE)$ we put $$ \cL(Y,\cE):= v_{!+}(\cE)$$ 
where $v: \bbY\rightarrow\bbP=(\cP_s,\cP_s)$ is the immersion of couples associated with $Y$ and $v_{!+}$ is the corresponding intermediate extension functor. We then have, cf. Thm. \ref{classRep}: 

\vskip5pt 

{\bf Theorem 1.}
{\it The correspondence $(Y,\cE)\mapsto H^0(\cP,\cL(Y,\cE))$ induces a bijection 
$$\{\text{equivalence classes of pairs $(Y,\cE)$}\}\car \{\text{irreducible $F$-holonomic $D^{\dagger}(\GG)_{\theta_0}$-modules} \}/ {\simeq} $$ }

For example, each couple $\bbY$ is equipped with the constant overconvergent $F$-isocrystal $\cO_{\bbY}$.
If $Z$ is a divisor in $\cP_s$ and $\cU=\cP \backslash Z$ with 
$\bbY=(\cU_s,\cP_s)$, then $\cO_{\bbY}=\cO_{\cP,\Q}(^\dagger Z)$, i.e. functions on $\cU$ with overconvergent singularities along $Z$.
 In general, if $Y$ 
admits a formal lift with connected rigid-analytic generic fibre, then $\cO_{\bbY}$ is irreducible and corresponds therefore to an 
irreducible $F$-holonomic $D^{\dagger}(\GG)_{\theta_0}$-module. 

\vskip5pt 

We expect that many $D^{\dagger}(\GG)_{\theta_0}$-modules, in particular those which come from admissible $G(K)$-representations, are in fact geometrically $F$-holonomic. As an example, we treat the case of highest weight modules (but there are many more, already in dimension one, cf. Theorem 3 below). We show that the central block of the classical BGG category $\cO_0$ embeds, via the base change $U(\frg)\rightarrow D^{\dagger}(\GG)$,
fully faithfully into the category of coherent $D^{\dagger}(\GG)$-modules (cf. Thm. \ref{ffembedding}). It is well-known that the irreducible modules in $\cO_0$ are parametrized by the Weyl group elements $w\in W$ via $L(w) := L(-w(\rho)-\rho)$ where $\rho$ denotes half the sum over the positive roots and where 
$L(-w(\rho)-\rho)$ denotes the unique irreducible quotient of the Verma module with highest weight $-w(\rho)-\rho$. We write $$L^{\dagger}(w) := D^{\dagger}(\GG) \otimes_{U(\frg)} L(w)$$ for its crystalline counterpart. On the other hand,  let $$Y_w:=BwB/B \subset P=G/B$$ be the Bruhat cell in $P$ associated with $w\in W$ and let $X_w$ be its Zariski-closure, a Schubert scheme. We have the couple $\bbY_w=(Y_{w,s},X_{w,s})$ and the immersion $v: \bbY_w\rightarrow \bbP.$ Our second main result is the following, cf. Thm. \ref{thm_final}:

\vskip5pt 

{\bf Theorem 2.}
{\it One has a canonical isomorphism of $\sD^{\dagger}_{\cP}$-modules
$$ {\mathscr Loc} (L^{\dagger}(w)) \simeq v_{!+}(\cO_{\bbY_w}).$$
In particular, the modules $L^{\dagger}(w)$ are geometrically $F$-holonomic for all $w\in W$. }

 \vskip5pt
 
 This result is in analogy with the classical result identifying the localization of the irreducible $U(\frg)$-module $L(w)$ with the intermediate extension of the constant connection on the complex Bruhat cell associated with $w$, cf. \cite{BB81,BK81}.  
 \vskip5pt 
 
 In the final section, we discuss somewhat detailed the example $G={\rm SL}_2$. In this case, $P$ equals the projective line over $\fro$ and {\it any}
 irreducible $\sD^{\dagger}_{\cP}$-module is holonomic. Moreover, theorem $1$ gives a classification in terms of irreducible 
overconvergent $F$-isocrystals $\cE$ on either a closed point $y$ of $\bbP_{k}^1$ or an open complement $Y$ of finitely many closed points $Z=\{y_1,...,y_n\}$ of $\bbP_{k}^1$. In the first case, the point is a complete invariant. For example, the point $y=\infty$ corresponds to the 
highest weight module $L^{\dagger}(-2\rho)$. In the second case, the empty divisor $Z=\varnothing$ corresponds to the trivial representation. 
For a non-empty $Z$, we may suppose that all its points $y_1,...,y_n$ are $k$-rational with $y_1=\infty$. There are then two extreme cases $$ Y=\bbA^1_k \hskip10pt \text{and} \hskip10pt Y= \bbP_{k}^1\setminus \bbP^1(k),$$
the affine line and Drinfeld's upper half plane, respectively. We illustrate the two by means of two "new" examples. 
In the case $Y=\bbA^1_k$ we assume that $K$ contains the $p$-th roots of unity $\mu_p$ and we choose an element $\pi\in\fro$ with ${\ord}_p (\pi) = 1/(p-1)$. We let ${\mathscr L}_{\pi}$ be the coherent $\sD^{\dagger}_{\cP}$-module defined by the {\it Dwork overconvergent $F$-isocrystal} on $\bbY$ associated with $\pi$. On the other hand, we let 
 $ \frn = K.e$ be the nilpotent radical of $Lie(B)$, where $e=\bigl( \begin{smallmatrix}
  0&1\\ 0&0
\end{smallmatrix} \bigr)$. Let $\eta: \frn \rightarrow K$ be a nonzero character and consider Kostant's {\it standard Whittaker module} 
 
 $$W_{\theta_0, \eta}:= U(\frg) \otimes_{ Z(\frg) \otimes U(\frn) } K_{\theta_0,\eta} $$

with character $\eta$ and infinitesimal character $\theta_0$ 
\cite[(3.6.1)]{KostantWhittaker}. It is an irreducible $U(\frg)$-module \cite[Thm. 3.6.1]{KostantWhittaker}, but {\it not} a highest weight module, i.e. it 
does not lie in $\cO_0$. We write $$W^{\dagger}_{\theta_0,\eta}:=D^{\dagger}(\GG)\otimes_{U(\frg) } W_{\theta_0,\eta}$$
for its crystalline counterpart. Our third main result is the following, cf. \ref{Dwork}: 
\vskip5pt

{\bf Theorem 3.} 
{\it  Consider the character $\eta$ defined by $\eta(e):=\pi$. 
There is a canonical isomorphism $${\mathscr Loc} (W^{\dagger}_{\theta_0,\eta} ) \car {\mathscr L}_{\pi} $$
as left $\sD^{\dagger}_{\cP}$-modules. In particular, $W^{\dagger}_{\theta_0,\eta}$ is geometrically $F$-holonomic.}
\vskip5pt 

The theorem shows, in particular, that the Dwork isocrystal ${\mathscr L}_{\pi}$ is {\it algebraic} in the sense that it comes from 
an algebraic $\sD_{\bbP^1_K}$-module, namely ${\rm Loc}(W_{\theta_0,\eta})$, by extension of scalars 
$\overline{\sD}_{\bbP^1_K}\rightarrow \sD^{\dagger}_{\cP}$. The holonomic $\sD_{\bbP^1_K}$-module  ${\rm Loc}(W_{\theta_0,\eta})$, however, 
is not regular, but has an irregular singularity at infinity. 

\vskip5pt 

We discuss an example in the Drinfeld case, where $Y= \bbP_k^1 \setminus \bbP^1(k)$. We identify $k=\bbF_q$.
We assume that $K$ contains the cyclic group $\mu_{q+1}$ of $(q+1)$-th roots of unity. The space $Y$ admits a distinguished unramified Galois covering 
 $u: Y' \rightarrow Y$ with Galois group $\mu_{q+1}$, given by the so-called {\it Drinfeld curve}

$$Y'=\Big\{(x,y) \in \bbA_k^2 \midc xy^q-x^qy=1\Big\}.$$

The latter admits a smooth and projective compactification $\overline{Y'}$. The covering map $u$ extends to a smooth and tamely ramified morphism 
$u: \overline{Y'} \rightarrow \bbP_k^1$ which maps the boundary bijectively to $Z=\bbP^1(k)$. We denote by $u: \bbY'\rightarrow\bbY$ the morphism of couples induced by $u$ in this situation and we let $$\cE=\bbR^{*}u_{\rm rig,*}\cO_{\bbY'}$$ be the relative rigid cohomology sheaf. Using results of Grosse-Kl\"onne \cite{GK_DL}, we show that $\cE$ admits an isotypic decomposition into overconvergent $F$-isocrystals $\cE(j)$ on $\bbY$ of rank one. In particular, each pair $(Y,\cE(j))$ corresponds in the classification of theorem $1$ 
to an irreducible geometrically $F$-holonomic $D^{\dagger}(\GG)_{\theta_0}$-module $H^0(\cP, v_{!+}\cE(j))$.

\vskip5pt 

We do not know whether the modules $H^0(\cP, v_{!+}\cE(j))$ are {\it algebraic}, in the sense that they arise by base change from irreducible $U(\frg)$-modules. If algebraic, to which class do they belong? We recall that irreducible $U(\frg)$-modules fall into three classes: highest weight modules, Whittaker modules and a third class whose objects (with a fixed central character) are in bijective correspondence with similarity classes of irreducible elements of a certain localization of the first Weyl algebra \cite{Block}. We plan to come back to these questions in future work. 
 
 \vskip5pt
 
{\it Notations and Conventions.} \label{notations}
In this article, $\fro$ denotes a complete discrete valuation ring of mixed characteristic $(0,p)$.
We let $K$ be its fraction field and $k$ its residue field, which is assumed to be perfect. 
We suppose that there exists a lifting of the Frobenius of $k$ to $\fro$.
We denote by $\vpi$ a uniformizer of $\fro$. All formal schemes $\frX$ over $\fro$ are assumed to be locally noetherian and such that $\vpi \cO_{\frX}$ is an ideal of definition. Without further mentioning, all occuring modules will be left modules. 

\section{Overholonomic modules and intermediate extensions}

For a smooth formal $\fro$-scheme $\frX$ we denote by $\sD^{\dagger}_{\frX}$ the sheaf of arithmetic differential operators on $\frX$.
We refer to \cite{BerthelotDI} for the basic features of the category of $\sD^{\dagger}_{\frX}$-modules.

\subsection{Overholonomic modules}\label{subsec_ovholext}

We introduce the framework of overholonomic complexes of arithmetic $\sD$-modules with Frobenius structure, following Abe-Caro \cite{AbeCaro}. 

\vskip5pt

Recall that a {\it frame} $(Y,X,\cP)$ is the data consisting of a separated and smooth formal scheme $\cP$ over $\fro$, a closed subvariety $X$ of its special fibre $\cP_s$, and an open subscheme $Y$ of $X$. A {\it morphism} between two such frames is the data $u=(b,a,f)$ consisting of morphisms 
$b: Y'\rightarrow Y, a: X'\rightarrow X, f: \cP'\rightarrow \cP$ such that $f$ induces $b$ and $a$. A {\it l.p. frame} $(Y,X,\cP,\cQ)$ is the data of a proper and smooth formal scheme $\cQ$ over $\fro$, an open formal subscheme $\cP\subset\cQ$ such that $(Y,X,\cP)$ is a frame. A {\it morphism} of l.p frames is defined in analogy to a morphisms of frames. It is called {\it complete} if the morphism $a: X'\rightarrow X$ is proper. 

A {\it couple} $\bbY$ is the data $(Y,X)$
consisting of a $k$-variety $X$ and an open subscheme $Y\subset X$ such that there exists a l.p. frame of the form $(Y,X,\cP,\cQ)$. A {\it morphism of couples} is 
the data $u=(b,a)$ consisting of morphisms $b: Y'\rightarrow Y, a: X'\rightarrow X$ such that $b$ is induced by $a$. It is called {\it complete} if $a$ is proper. 
Let P be a property of morphisms of schemes. One says that $u$ is $c$-P if $u$ is complete and $b$ satisfies P.
For all this, cf. \cite[1.1]{AbeCaro}.

\vskip5pt 
Denote by $\cP$ a smooth and proper formal scheme over $\fro$.

\vskip5pt

We denote by $D^b_{\rm ovhol}(\sD^{\dagger}_{\cP})$ the
triangulated category of complexes of overholonomic $\sD^{\dagger}_{\cP}$-modules introduced by Caro \cite[3.1]{Caro_ovhol}.
Let $Z$ be a closed subset of $\cP_s$, the special fibre of $\cP$. There are two functors $R\Gamma^{\dagger}_Z$ and $(^\dagger Z)$
defined on $D^b_{\rm ovhol}(\sD^{\dagger}_{\cP})$ giving rise to a localization triangle

$$ R\Gamma^{\dagger}_Z(\cE) \ra \cE \ra (^\dagger Z)\cE \stackrel{+1}{\ra } $$
for $\cE\in D^b_{\rm ovhol}(\sD^{\dagger}_{\cP})$, cf. \cite[1.1.5]{AbeCaro}. 

\vskip5pt 
 Let now $\bbY=(Y,X)$ be a couple such that 
$(Y,X,\cP,\cP)$ is a l.p. frame.
By abuse of notation, we will sometimes denote the frame $(Y,X,\cP)$ (or even the l.p. frame $(Y,X,\cP,\cP)$) by $\bbY$, too. This should not cause confusion. 

\vskip5pt 

The couple $\bbP$ is obtained by taking $Y=X=\cP_s$. 

\vskip5pt

Let $Z=X \backslash Y$. For $\cE\in D^b_{\rm ovhol}(\sD^{\dagger}_{\cP})$ one sets
$$ \bbR\Gamma^{\dagger}_Y(\cE):=\bbR\Gamma^{\dagger}_X\circ (^\dagger Z)(\cE).$$
The category $D^b_{\rm ovhol}(\bbY/K)$ of overholonomic complexes of arithmetic $\sD$-modules on $\bbY$
is defined to be the full subcategory of $D^b_{\rm ovhol}(\sD^{\dagger}_{\cP})$ formed by objects $\cE$ such that there is an isomorphism 
 $\cE\car \bbR\Gamma^{\dagger}_Y(\cE)$ \cite[1.1.5]{AbeCaro}. Of course,
$D^b_{\rm ovhol}(\bbP/K)=D^b_{\rm ovhol}(\sD^{\dagger}_{\cP})$. 

\vskip10pt

Abe-Caro introduce a canonical $t$-structure on $D^b_{\rm ovhol}(\bbY/K)$ in the following way \cite[1.2]{AbeCaro}. Write
$\cU=\cP\setminus Z$ for the open complement of $Z$ in $\cP$. Then $D^{\geq 0}_{\rm ovhol}(\bbY/K)$ is defined to be the strictly full
subcategory of objects $\cE\in D^b_{\rm ovhol}(\bbY/K)$ such that 
$$\cE_{|\cU}\in {D^{\geq 0}(D^{\dagger}_{\cU}})$$ (analogously for $\leq 0$).
The trunctation functors relative to the couple $\bbY$ are defined to be $$\tau^{\bbY}_{\geq 0}=(^\dagger Z)\circ \tau_{\geq 0}\hskip8pt \text{resp.} \hskip8pt\tau^{\bbY}_{\leq 0}=(^\dagger Z)\circ \tau_{\leq 0}, $$
where $\tau_{\geq 0}$ resp. $\tau_{\leq 0}$ are the usual truncation functors. The functors $\tau^{\bbY}_{\geq 0}$ and $\tau^{\bbY}_{\leq 0}$ define 
a $t$-structure on $F$-$D^b_{\rm ovhol}(\bbY/K)$ whose heart is denoted by $F\text{-Ovhol}(\bbY/K)$ \cite[1.2.9]{AbeCaro}. 
The latter is an abelian category which is noetherian and artinian \cite[1.2.13]{AbeCaro}. When $Y$ is smooth, the category $F\text{-Ovhol}(\bbY/K)$
contains a full subcategory $F$-${\rm Isoc}^{\dagger\dagger}(\bbY/K)$ which is equivalent to the category of overconvergent $F$-isocrystals on $\bbY$, the usual coefficients of rigid cohomology \cite[1.2.14]{AbeCaro}.
\vskip5pt

We recall an important key lemma. 
\begin{lemma}\label{LemFond} Let
$\cU=\cP\setminus Z$ and let $\alpha$ be a morphism in $D^b_{\rm ovhol}(\bbY/K)$. 
Then $\alpha$ is an isomorphism in $D^b_{\rm ovhol}(\bbY/K)$ if and only if $\alpha_{|\cU}$ is an isomorphism in ${D^{b}(D^{\dagger}_{\cU}})$.
\end{lemma}
\begin{proof} This is \cite[1.2.3]{AbeCaro}.
\end{proof}

\vskip5pt 

Main examples: (i) In the case where $\bbY=\bbP$, the category $F\text{-Ovhol}(\bbP/K)$ is the usual category of overholonomic arithmetic $F$-$\sD^{\dagger}_{\cP}$-modules on $\cP$. 

(ii) If $Z$ is a divisor in $\cP_s$ with open complement $Y=\cP_s \backslash Z$ and 
$\bbY=(Y,\cP_s,\cP)$, then  $F\text{-Ovhol}(\bbY/K)$ is the usual category of overholonomic 
$F$-$\sD^{\dagger}_{\cP}(^\dagger Z)$-modules.

\subsection{Intermediate extensions}

We keep the notation of the previous subsection. 
We introduce the intermediate extension functor for arithmetic $\sD$-modules following Abe-Caro \cite{AbeCaro}. 

\vskip5pt

Let $$u: \bbY \longrightarrow \bbY'$$ be a complete morphism of couples. 
There is a canonical homomorphism $$\theta_{u,\cE}: u_! \cE \longrightarrow u_+\cE $$ for any complex $\cE\in F$-$D^b_{\rm ovhol}(\bbY)$, cf. \cite[1.3.4]{AbeCaro}.
The morphism is compatible with composition in the following sense:  if $w=u_2\circ u_1$, where $u_1$ and $u_2$ are 
c-complete morphisms of couples, then

\begin{gather}\label{compatib}
\xymatrix{ u_{2!}\circ u_{1!}\cE  \ar@/_2pc/[rr]^{\theta_{u_2\circ u_1}}\ar@{->}[r]^{u_{2!}(\theta_{u_1})}& u_{2!}\circ u_{1+}\cE
\ar@{->}[r]^{\theta_{u_2}(u_{1+})} & u_{2+}\circ u_{1+}\cE
} 
\end{gather}
by \cite[Prop. 1.3.7]{AbeCaro}.
We denote by an exponent $(-)^0=\cH^0_t(-)$ the application of the first cohomology sheaf $\cH^0_t= \tau^{\bbY}_{\leq 0} \tau^{\bbY}_{\geq 0}$ relative to the $t$-structure on $F$-$D^b_{\rm ovhol}(\bbY/K)$ (and similar for $\bbY'$).
If $u$ is a c-immersion,
and if $\cE\in F$-$\text{Ovhol}(\bbY)$, then the {\it intermediate extension of $\cE$ on $\bbY'$} is defined to be
$$   u_{!+}(\cE):=\im(\theta^0_{u,\cE}\,\colon \, u_{!}^0\cE \ra u_{+}^0\cE).$$
Note that  if $w$ is a c-affine immersion, then  $u_+$ and $u_!$ are $t$-exact 
 by \cite[Remark 1.4.2]{AbeCaro}, so that the definition simplifies to 
$$   u_{!+}(\cE)=\im(\theta_{u,\cE}\,\colon \, u_{!}\cE \ra u_{+}\cE).$$

\vskip5pt

\subsection{A classification result} \label{classify} We keep the notation of the previous subsections. In particular, 
$\cP$ still denotes a smooth and proper formal scheme over $\fro$. Our goal here is to classify the 
irreducible overholonomic $F$-$\sD^{\dagger}_{\cP}$-modules, up to isomorphism. 
This is in close analogy to the classical setting of algebraic $D$-modules on complex varieties, e.g. \cite[3.4]{Hotta}. 

\vskip5pt 

We will only
consider couples that arise from a smooth locally closed subvariety $Y\subseteq \cP_s$ by taking its Zariski closure $X=\overline{Y}$ 
in $\cP_s$. Then $(Y,X,\cP)$ is a frame and $(Y,X,\cP,\cP)$ is a l.p. frame and $\bbY=(Y,X)$ is a couple.
By abuse of notation, we will sometimes denote the frame $(Y,X,\cP)$ (or even the l.p. frame $(Y,X,\cP,\cP)$) by $\bbY$, too. This should not cause confusion.

\vskip5pt 

For such a couple $\bbY=(Y,X)$, 
we consider the corresponding c-locally closed immersion  
$$v: \bbY \longrightarrow \bbP.$$ The associated intermediate extension functor 
$$v_{!+}:  F\text{-Ovhol}(\bbY/K)\longrightarrow  F\text{-Ovhol}(\bbP/K)= \{ \text{ overholonomic $F$-}\sD^{\dagger}_{\cP}\text{-modules }\}$$

is given by

$$ v_{!+}(\cE) := {\rm Im}\big( \theta_{v,\cE}^0 : v_{!}^0\cE\longrightarrow v_{+}^0\cE\big).$$

\vskip5pt

Suppose for a moment that $Y \subset\cP_s$ is {\it closed} and there exists a $\fro$-smooth closed formal 
subscheme $\frY\subset\cP_s$, defined by some coherent ideal sheaf in $\cO_{\cP}$, which lifts the closed immersion $Y\subset\cP_s$.
Then $\text{Ovhol}(\bbY)$ identifies with the category of overholonomic $\sD^{\dagger}_{\cY}$-modules and the functor $v_{!+}$
coincides with the direct image functor appearing in Kashiwara's equivalence \cite{BerthelotIntro,HS3}. By the latter equivalence, 
the functor $v_{!+}$ induces a bijection between the (isomorphism classes of) irreducible $\sD^{\dagger}_{\cY}$-modules and 
irreducible $\sD^{\dagger}_{\cP}$-modules supported on $\cY$.

\vskip5pt 
 
 The case of a closed immersion generalizes as follows. 

\begin{lemma} 
Let $\cM$ be an irreducible overholonomic $F$-$\sD^{\dagger}_{\cP}$-module. There is an open dense smooth subscheme $U\subset \cP_s$ 
with the property: if $u: \bbU=(\cU,\cP_s,\cP)\rightarrow\bbP$ denotes the corresponding c-open immersion, then $u^{!} \cM$ is an overconvergent
$F$-isocrystal on $\bbU$.
\end{lemma}
\begin{proof} By \cite[1.4.9(i)]{AbeCaro}, we know that there is an open dense smooth subscheme $U\subset \cP_s$ and an overconvergent
$F$-isocrystal $\cE$ on $\bbU$ such that $\cM=u_{!+}\cE$. By left $t$-exactness of $u^{!}$ we obtain $u^{!} \cM\subset u^{!}u^0_{+}\cE =\cE$.
Hence, $u^{!} \cM$ is an overconvergent $F$-isocrystal.
\end{proof}

Since any overholonomic $\sD^{\dagger}_{\cP}$-module $\cM$ is coherent \cite[3.1]{Caro_ovhol}, we may view 
its support $\rm{Supp}(\cM)$ as a closed (reduced) subvariety of $\cP_s$. 

\begin{prop} \label{prop-surjective}
Let $\cM$ be an irreducible overholonomic $F$-$\sD^{\dagger}_{\cP}$-module.
There is an open dense smooth affine subscheme of an irreducible component of 
$\rm{Supp}(\cM)$ with the property: if $v: \bbY\rightarrow\bbP$ denotes the corresponding immersion, then $\cE:=v^{!}\cM$ is an irreducible overconvergent $F$-isocrystal on $\bbY$. Moreover, $v_{!+}(\cE) =\cM$. 
\end{prop}
\begin{proof} According to the preceding lemma, we may choose an open dense subscheme $U\subset \cP_s$ over which $\cM$ becomes 
an overconvergent isocrystal. We may choose an open dense smooth affine subscheme $Y$ of an irreducible component of $\rm{Supp}(\cM)$, which is contained in $U$. Let 
$\cE:=v^{!}\cM$. If $k$ denotes the c-closed immersion $\bbY\rightarrow\bbU$, then Abe-Caro's version of Kashiwara's theorem \cite[1.3.2(iii)]{AbeCaro} together with \cite[1.4.9(ii)]{AbeCaro} imply that $\cE= k^{!}u^{!}\cM$ is irreducible. Moreover, by adjointness \cite[1.1.10]{AbeCaro} 
$$ \Hom ( v_{!}\cE,\cM) =  \Hom ( \cE, v^{!}\cM) \neq 0$$
and there is therefore a non-zero morphism $v_{!}\cE\rightarrow \cM$. In other words, $\cM$ is a quotient of $v_{!}\cE$. 
But $v_{!}\cE=v^0_{!}\cE$, since $Y$ is affine, and $v_{!+}\cE$ is the unique irreducible quotient of $v^0_{!}\cE$ \cite[1.4.7(ii)]{AbeCaro}.
We therefore must have $v_{!+}\cE=\cM$. 
\end{proof}

Consider now a pair $(Y,\cE)$ where $Y\subset \cP_s$ is a connected smooth locally closed subvariety and $\cE$ is an irreducible overconvergent $F$-isocrystal on $\bbY=(Y,X)$. We write $$\cL(Y,\cE):= v_{!+}(\cE)\in F\text{-Ovhol}(\bbP).$$ 

\begin{prop} \label{prop-interext} The overholonomic $F$-$\sD^{\dagger}_{\cP}$-module $\cL(Y,\cE)$ is irreducible, has support $\overline{Y}$ and satisfies 
$v^{!} \cL(Y,\cE)=\cE$.
\end{prop}
\begin{proof}
The irreducibility statement and the fact that $0\neq v^{!} \cL(Y,\cE)\subset v^{!}v^0_{+}\cE = \cE$ follow from \cite[1.4.7(i)]{AbeCaro} and its proof. 
Since $\cE$ is irreducible, $v^{!} \cL(Y,\cE)=\cE$ as claimed. Finally, 
if $k: \bbY\rightarrow\bbU$ is a $c$-closed immersion and $u: \bbU\rightarrow\bbP$ a $c$-open immersion such that $v=u\circ k$, then
$v_{!+}=u_{!+}\circ k_{!+}$ \cite[1.4.5(i)]{AbeCaro}. The support of $k_{!+}\cE=k_{+}\cE$ equals $Y$ and 
the support of $ \cL(Y,\cE)=u_{!+}k_{!+}\cE$ equals $\overline{Y}$. 
\end{proof}

Two pairs $(Y,\cE)$ and  $(Y',\cE')$ are said to be {\it equivalent} if $\overline{Y}=\overline{Y'}$ and there is an open dense
$U\subset \overline{Y}$ contained in the intersection $Y\cap Y'$ such that $u^{!}\cE \simeq u'^{!}\cE'$. Here 
$u$ denotes the $c$-open immersion $\bbU=(U,\overline{Y}, \cP) \rightarrow \bbY$ and similarly for $u'$. This defines an equivalence relation on the set of couples.

\begin{thm}\label{classification} The correspondence $(Y,\cE)\mapsto \cL(Y,\cE)$ induces a bijection 
$$\{\text{equivalence classes of pairs $(Y,\cE)$}\}\car \{\text{irreducible overholonomic $F$-$\sD^{\dagger}_{\cP}$-modules} \}/ {\simeq} $$
\end{thm}
\begin{proof}
Let us show that the map in question is well-defined.
Let  $(Y,\cE)$ and $(Y',\cE')$ be two equivalent couples. Choose an open dense
$U\subset \overline{Y}$ contained in the intersection $Y\cap Y'$ such that $u^{!}\cE \simeq u'^{!}\cE'$. Note that $v\circ u = v' \circ u'$.
Define $\cF= (v\circ u)^{!}\cL(Y,\cE)$ and similarly for $\cL(Y',\cE')$.
Then $(v\circ u)_{!+} \cF = \cL(Y,\cE)$ according to  \ref{prop-surjective} and $\cF= u^{!} v^{!} \cL(Y,\cE) = u^{!} \cE$ according to \ref{prop-interext}.
Hence, $\cF\simeq \cF'$ and we obtain $\cL(Y,\cE) \simeq \cL(Y',\cE')$. 

Let us next show that the map is injective. So suppose that $\cL(Y,\cE) \simeq \cL(Y',\cE')$ for two couples $(Y,\cE)$ and  $(Y',\cE')$. Then \ref{prop-interext} implies $\overline{Y}=\overline{Y'}$ and moreover, if $U\subset \overline{Y}$ is open dense and contained in the intersection $Y\cap Y'$, then $(v\circ u)^{!} \cL(Y,\cE)= u^{!}\cE$. Since $v\circ u = v' \circ u'$, we obtain $u^{!}\cE \simeq u'^{!}\cE'$ as desired.
This proves the injectivity. The surjectivity of the map is a direct consequence of  \ref{prop-surjective}.
\end{proof}

\vskip10pt

Let $Y\subset \cP_s$ be a smooth locally closed subvariety and $\bbY=(Y,X)$. 
\begin{dfn} \label{def_constant} Let $d:=\dim(\cP_s)-\dim(Y)$. We define the {\it constant overholonomic module} on the frame $\bbY$ to be 
         $$ \cO_{\bbY}=R\Gamma_{\bbY}(\cO_{\cP,\Q})[d].$$
\end{dfn}

\begin{prop}\label{constirred} Suppose that $Y$ is connected and there exists a smooth formal 
scheme $\frY$ over $\fro$, so that the immersion $Y \ra \cP$ lifts to some morphism of formal schemes
$\cY \hra \cP$. The module $\cO_{\bbY}$ lies in $F$-${\rm Isoc}^{\dagger\dagger}(\bbY/K)$.
If the rigid-analytic generic fiber $\cY_K$ is connected,
then $\cO_{\bbY}$ is an irreducible object in the category 
$ {\rm Ovhol}(\bbY)$.
\end{prop}
\begin{proof} 
Denote $Z=X \backslash Y$ and
$\cU= \cP\backslash Z$. We have the closed immersion of smooth formal schemes 
$v: \cY \hra \cU$. Then, by ~\cite[Proposition 1.4]{Berthelot-cohdif}, we see that 
$$\cO_{\bbY|\cU} = R\Gamma_{\cY}(\cO_{\cU,\Q})[d]\simeq  v_+v^{!}\cO_{\cU,\Q}[d]  = v_+\cO_{\cY}.$$
This coincides with ${\rm sp}_{+}\cO_{\bbY}$ and hence lies in the category $F$-${\rm Isoc}^{\dagger\dagger}(Y,\cU/K)$,
in the notation of \cite[1.2.14]{AbeCaro}. This shows $\cO_{\bbY}\in F$-${\rm Isoc}^{\dagger\dagger}(\bbY/K)$.
\vskip5pt 
The irreducibility statement is based on the following
\begin{lemma} \begin{enumerate}  Let $\cQ$ be a connected smooth formal scheme over $\fro$ and $\cQ_K$ its 
generic fiber (as rigid analytic space). Assume furthermore that $\cQ_K$ is connected.
\item The constant isocrystal $\cO_{Q_K}$ is irreducible in the category of convergent isocrystals.
\item The coherent $\sD^{\dagger}_{\cQ}$-module $\cO_{\cQ,\Q}$ is irreducible in the category of
$\sD^{\dagger}_{\cQ}$-modules.
 \end{enumerate}
\end{lemma}
\begin{proof} We begin by (i). 
Let $E$ be a subobject of $\cO_{Q_K}$ in the abelian category of convergent isocrystals over $\cQ_K$, and 
$E'=\cO_{\cQ_{K}}/E$ be the quotient. 
As convergent isocrystals over $\cQ_K$, $E$ and $E'$ are
locally free $\cO_{\cQ_K}$-modules so that there exists an admissible cover by affinoids $\cU_i$ ($i\in I$) such that 
$E_{|\cU_i}$ and $E'_{|\cU_i}$ are free $\cO_{\cU_i}$-modules for each $i$. Fix $i_0$ and denote by 
$A=\Ga(\cU_{i_0},\cO_{\cU_{i_0}})$. Since $\cU_{i_0}$ is affinoid, we have an exact sequence of free $A$-modules 
$$   0 \ra \Ga(\cU_{i_0},E) \ra A \ra \Ga(\cU_{i_0},E')\ra 0.$$ Take $x$ a point of $\cU_{i_0}$, and $K(x)$ its residue
field, then the previous exact sequence remains exact after tensoring by $K(x)$, meaning that
 $\Ga(\cU_{i_0},E)$ is
either equal to $0$ or to $A$. Assume for example that this is equal to $0$, so that $E_{|\cU_{i_0}}=0$ by 
Tate's acyclicity theorem. By Zorn's lemma there is a maximal subset 
$J\subset I$ such that $E_{|U_i}=0$ for each $i \in J$. Assume that $J\not =I$ then $J'=I \setminus J$ is not empty. By
connectedness, the union $\bigcup_{i\in J'}U_i$ intersects the union $\bigcup_{i\in J}U_i$, thus there exist $l\in J'$, 
$i \in J$ such that $U_l \bigcap U_i \neq \emptyset.$ Since $E_{|U_l}$ is either equal to $0$ or to 
$\cO_{U_l}$, we see that it is zero by restricting to $U_l \bigcap U_i$, which contradicts the fact that 
$J\neq I$. This proves (i).

For (ii) we use then that the abelian category of convergent isocrystals over the generic fiber $\cQ_K$ of the
formal scheme $\cQ$ is equivalent to the category of coherent $\sD^{\dagger}_{\cQ}$-modules, that are coherent
$\cO_{\cQ,\Q}$-modules (\cite[4.1.4]{BerthelotDI}). The functors $sp_*$ and $sp^*$ realize this equivalence of
categories.  Let $\cE$ be a non-zero coherent $\sD^{\dagger}_{\cQ}$-submodule of
$\cO_{\cQ,\Q}$, then $E=sp^*\cE$ is a convergent isocrystal, that is a subobject of $\cO_{Q_K}$. By (i), it is either
$0$ or equal to the constant convergent isocrystal $\cO_{\cQ_K}$. Thus $\cE$ is either $0$ or $\cO_{\cQ,\Q}$ and this proves (ii).
\end{proof}

Let us come back to the proof of the proposition. Let $\alpha: \cE \hra \cO_{\bbY}$ be an injective
morphism in the category $\text{Ovhol}(\bbY)$. As remarked in the beginning of the proof, 
$$\cO_{\bbY|\cU} = R\Gamma_{\cY}(\cO_{\cU,\Q})[d]\simeq  v_+v^{!}\cO_{\cU,\Q}[d]  = v_+\cO_{\cY}.$$
By Kashiwara's theorem for the closed immersion $v: \cY \hra \cU$
 \cite{BerthelotIntro,HS3} and the previous lemma, $v_+\cO_{\cY}$ is irreducible in the category of coherent
$\sD^{\dagger}_{\cU}$-modules with support in $\cY$, so that $\cE_{|\cU}$ is either $0$ or equal to $v_+\cO_{\cY}$. 
Using ~\ref{LemFond}, we conclude that $\cE$ is either $0$ or equal to $\cO_{\bbY}$.
\end{proof}

Example: If $Z$ is a divisor in $\cP_s$ , $\cU=\cP \backslash Z$, 
$\bbY=(\cU_s,\cP_s,\cP)$ then  $F-\text{Ovhol}(\bbY)$ is the usual category of overholonomic $F-\sD^{\dagger}_{\cP}(^\dagger
Z)$-modules. In this case, if $\cU$ and its generic fiber $\cU_K$ are connected,
then the constant overholonomic module $\cO_{\bbY}=\cO_{\cP,\Q}(^\dagger Z)$ is an irreducible 
$\sD^{\dagger}_{\cP}(^\dagger Z)$-module by the previous proposition applied to $Y=\cU_s$, the special fiber of $\cU$.

\begin{prop} The intermediate extension  
$v_{!+}(\cO_{\bbY})$
is an irreducible overholonomic $F$-$\sD^{\dagger}_{\cP}$-module.
\end{prop}
\begin{proof}
This follows from the theorem \ref{classification} and the above proposition.
\end{proof}

\section{Some compatibility results between generic and special fibre}

We keep the notations introduced in the preceding section. 
In this section, we place ourselves into certain integral situations involving schemes over $\fro$ and establish
various compatibilities between the classical intermediate extensions on generic fibres and Abe-Caro intermediate extensions 
arising after reduction on the special fibre. We will focus in particular on the cases of open immersions and proper morphisms. 

\vskip5pt 

The results of this section are then applied in the final section \ref{secfinal} in the case of the flag variety, in order to 
compare intermediate extensions over the Bruhat cells, both in the generic and in the special fibre, cf. prop. \ref{app_flag} and thm. \ref{thm_final}.

\subsection{Notations}
Let us begin with some notations: if $X$ denotes a $\fro$-scheme, then $X_s$ will be its special fiber, 
$X_{\bbQ}$ its generic fiber, $X_i=X\times \Spec{\fro/\varpi^{i+1}}$, $\frX$ (or $\cX$) the associated formal scheme obtained after $p$-adic 
completion, and $\bbX$ is the frame $\bbX=(X_s,X_s,\frX)$. Moreover, if $f: X \ra Y$ is a morphism 
of $\fro$-schemes, then $f_s$ (resp. $f_{\bbQ}$, 
$\hf,f_i$) will denote the induced morphism $X_s \ra Y_s$ (resp. $X_{\bbQ}\ra Y_{\bbQ}$, $\frX \ra \frY$, 
$X_i \ra Y_i$), 
and $F=(f_s,f_s,\hf)$ will denote the morphism of frames between the frames $\bbX$ and $\bbY$. 

\subsection{Open immersions} \label{debcompat}
Let $P$ be a smooth scheme over $\fro$. The closed immersions $P_i \hra \cP$ give rise to a canonical morphism of ringed spaces 
$$  \alpha \,\colon \, \cP =\varinjlim_i P_i \ra P.$$
It comes with the diagram
$$ \cP \sta{\alpha}{\ra} P \sta{j}{\hla} P_{\bbQ}$$ 

which will be our basic underlying structure in the following. We have the following first result.

\begin{lemma} \begin{enumerate} \item There is a canonical isomorphism $$\cO_{P,\bbQ}\simeq j_*\cO_{P_{\bbQ}}.$$
                \item There is a canonical isomorphism $$\sD^{(m)}_{P,\bbQ}\simeq j_*\sD_{P_{\bbQ}}$$ for any  $m$.
              \end{enumerate}
\end{lemma}
\begin{proof} We consider the canonical morphism of quasi-coherent 
$\cO_{P}$-sheaves $\cO_{P} \ra j_*\cO_{P_{\bbQ}}$. After tensoring with $\bbQ$, we get a morphism
$\cO_{P,\bbQ} \ra j_*\cO_{P_{\bbQ}}$. If $P=\Spec A$ is affine, then this morphism is the identity of
$A_{\bbQ}=\Ga(P,\cO_{P,\bbQ})=\Ga(P,j_*\cO_{P_\bbQ})$. 
This proves (i). For (ii), we start with the canonical morphism 
$$\sD^{(m)}_{P} \ra j_*\sD_{P_{\bbQ}}\simeq j_*\sD^{(m)}_{P,{\bbQ}}.$$ It induces a morphism $\sD^{(m)}_{P,\Q} \ra j_*\sD_{P_{\bbQ}}$. In order to prove that this morphism is an isomorphism, we may assume that 
$P$ is affine with local coordinates $x_1,\ldots,x_M$. In this case, both sheaves are free 
$\cO_{P,\bbQ}$-modules with basis $\uder^{\uk}$ and we conclude using (i).
\end{proof}

For a quasi-coherent $\cO_{P_{\bbQ}}$-module $\cE$, we set
$$ \ocE:=\alpha^{-1}j_{*}\cE.$$ 
\begin{lemma} The formation $\cE \mapsto \ocE$ is an exact functor from the category of quasi-coherent $\cO_{P_{\bbQ}}$-modules 
to the category of $\ocO_{P_{\bbQ}}$-modules. 
\end{lemma}
\begin{proof}
This statement comes from the fact that the functor $j_*$ is exact on quasi-coherent $\cO_{P_{\bbQ}}$-sheaves, since $j$
is affine, as well as $\alpha^{-1}$. The functor $\cE \mapsto \ocE$ is thus exact as the composition of two exact
functors.
\end{proof}

According to the lemma, the functor $\cE \mapsto \ocE$ passes directly to derived categories and gives a functor 
$$D^b_{qcoh}(\cO_{P_{\bbQ}})\longrightarrow D^b(\ocO_{P_{\bbQ}}).$$

We now consider the sheaf of rings $\osD_{P_{\bbQ}}$ on $\cP$.

\begin{lemma} \label{platitude}There is an injective flat morphism of sheaves of rings $$\osD_{P_{\bbQ}} \hrig  \sD^{\dagger}_{\cP}.$$
\end{lemma}
\begin{proof} If $U\subset P$ is an open affine of $P$ with local coordinates $x_1,...,x_M$, then  
the following description 
$$ \Ga(U,\alpha_*(\sD^{\dagger}_{\cP})) =\left\{\sum_{\unu}  a_{\unu}\uder^{[ \unu ]}\,|\, a_{\unu}\in 
\cO_{\frU}\otimes\bbQ\, | \,
\exists c>0, \eta<1,\,||a_{\unu}||\leq c\eta^{|\unu|}\right\}$$ and 
$$ \Ga(U,\sD_{P,{\bbQ}}) =\left\{\sum_{\unu, finite}  a_{\unu}\uder^{[ \unu ]} \,|\, a_{\unu}\in 
\cO_{U}\otimes\bbQ \right\}.$$
This gives the inclusion. Next, the ring $\sD^{\dagger}_{\cP}$ is flat over 
$\hsD_{\cP,\bbQ}^{(0)}$ \cite[Cor. 3.5.4]{BerthelotDI}. Moreover the sheaf $\hsD_{\cP}^{(0)}$
 is flat over $\sD_{P}^{(0)}$, by completion, 
so that $\sD^{\dagger}_{\cP}$ is indeed flat over $\osD_{P_{\bbQ}}$. 
\end{proof}
The proof of the following lemma is easy and left to the reader.
\begin{lemma}\begin{enumerate} \item Let $\cE$ be a coherent $\sD_{P_{\bbQ}}$-module, then
$\sD^{\dagger}_{\cP}\ot_{\osD_{P_{\bbQ}}}\ocE$ is a coherent $\sD^{\dagger}_{\cP}$-module.
\item Let $\cE\in D^b_{coh}(\sD_{P_{\bbQ}})$, then 
     $\sD^{\dagger}_{\cP}\ot_{\osD_{P_{\bbQ}}}\ocE\, \in D^b_{coh}(\sD^{\dagger}_{\cP}).$
\end{enumerate}
\end{lemma}

\vskip5pt
The following proposition is due to Virrion and shows that duality commutes with scalar extension. 
Her formulation involves perfect complexes, but since the scheme $P_{\bbQ}$ is smooth, the sheaf 
$\osD_{P_{\bbQ}}$ has finite cohomological dimension and the category $D^b_{coh}(\osD_{P_{\bbQ}})$ coincides with the 
category of perfect complexes of $\osD_{P_{\bbQ}}$-modules.
\begin{prop}\label{compat_dual} Let $\cE\in D^b_{coh}(\sD_{P_{\bbQ}})$,
then 
  $$\sD^{\dagger}_{\cP}\ot_{\osD_{P_{\bbQ}}}\bbD_{\osD_{P_{\bbQ}}}(\ocE) \simeq 
                            \bbD_{\sD^{\dagger}_{\cP}}(\sD^{\dagger}_{\cP}\ot_{\osD_{P_{\bbQ}}}\ocE) .$$
\end{prop}
\begin{proof} This is \cite[1.4,4.4]{Virrion-dual}.
\end{proof}
Let us recall that, if $d=\dim(P_{\bbQ})$, 
$$\bbD_{\sD_{P_{\bbQ}}}(\cE)= R \cH om_{\sD_{P_{\bbQ}}}(\cE,\sD_{P_{\bbQ}}[d])
\ot_{\cO_{P_{\bbQ}}}\omega_{P_{\bbQ}}.$$

\vskip5pt
\begin{dfn}\label{deftrans} Let $P$ be a smooth $\fro$-scheme and $Z\subset P$ a divisor. We say that $Z$ 
is a {\it transversal divisor} if $Z_s$ and $Z_{\bbQ}$ are divisors respectively of $P_s$ and $P_{\bbQ}$.
\end{dfn}

Let $Z$ be a transversal divisor and let $j: P\backslash Z \hookrightarrow P$ be the inclusion of its open complement. 
For any coherent $\sD_{P_{\bbQ}}$-module $\cE$, we put
$$(*Z_{\bbQ})\cE:=\sD_{P_{\bbQ}}(*Z_{\bbQ})\ot_{\sD_{P_{\bbQ}}}\cE$$
so that $(*Z_{\bbQ})\cE=j_{\Q+} j_{\Q!}\cE$. In the same spirit, we define for any coherent
$\sD^{\dagger}_{\cP}$-module $\cE$,
$$(^{\dagger}Z_s)\cE:= \sD^{\dagger}_{\cP}(^{\dagger}Z_s)\ot_{\sD^{\dagger}_{\cP}}{\cE}.$$
Let $Y=P\backslash Z$ with immersion $j: Y \hra P$ and consider the morphism of frames $$J: \bbY:=(Y_s,P_s,\cP)\ra \bbP:=(P_s,P_s,\cP).$$ 
Then $(^{\dagger}Z_s)\cE=J_+J^!\cE$. Note that, in this situation, $J_+$ is just the forgetful functor from the category 
$\text{Ovhol}(\bbY/K)$ to the category $\text{Ovhol}(\bbP/K)$. Moreover the functor $j_{\Q+}$ is exact since $Z_{\Q}$ is a divisor of $P_{\Q}$,
 and induces an equivalence of categories between coherent 
$\sD_{P_{\Q}}(*Z_{\Q})$-modules and coherent  $\sD_{Y_{\Q}}$-modules. Of course, at the level of sheaves of 
$\cO_{Y_{\Q}}$-modules, we have $j_{\Q+}=j_{\Q*}$. Recall also
that, in this situation, objects of $\text{Ovhol}(\bbY/K)$ consist of degree zero complexes of
$\sD^{\dagger}_{\cP}(^{\dagger}Z_s)$-modules by ~\cite[Remark 1.2.7 (iii)]{AbeCaro}.
\begin{prop} \label{bidual} Let $\cE\in D^b_{hol}({Y_{\bbQ}})$ and suppose that 
$\cF:=\sD^{\dagger}_{\cP}\ot_{\osD_{P_{\bbQ}}}\overline{j_{\Q +}\cE}\in D^b_{ovhol}({\bbY}).$ Then there is a commutative diagram
$$\xymatrix{ \overline{j_{\Q+}\cE} \ar@{->}[r]^{\overline{j_{\Q+} c_{\Q}}} \ar@{->}[d]^{1 \ot id_{j_{\Q+}\cE}} &
\overline{j_{\Q +}\bbD_{Y_{\Q}}\circ\bbD_{Y_{\Q}}(\cE)}\ar@{->}[d]^{}_{} \\
 J_+(\sD^{\dagger}_{\cP}\ot_{\osD_{P_{\bbQ}}}\overline{j_{\Q +}\cE}\ar@{->}[r]^>>>>>{J_+C}) &
  J_+\bbD_{\bbY}\circ\bbD_{\bbY}\left(\sD^{\dagger}_{\cP}\ot_{\osD_{P_{\bbQ}}}
\overline{j_{\Q +}\cE}\right) 
. }$$

Here,
$c_{\Q}$ is the canonical isomorphism
${\cE}\simeq \bbD_{Y_{\Q}}\circ\bbD_{Y_{\Q}}(\cE)$  and $C$ is the canonical isomorphism
${\cF}\simeq \bbD_{\bbY}\circ\bbD_{\bbY}(\cF)$. 
\end{prop}
\begin{proof} Let us first remark that the sheaf $\sD^{\dagger}_{\cP}(^{\dagger}Z_s)$ is flat over 
$\sD^{\dagger}_{\cP}$ and thus flat over $\osD_{P_{\bbQ}}$. This explains that no derived tensor product appears 
in the stated diagram. Moreover it will be enough to prove the statement for 
a single holonomic $\sD_{Y_{\Q}}$-module $\cE$ such that 
$\cF:=\sD^{\dagger}_{\cP}\ot_{\osD_{P_{\bbQ}}}\overline{j_{\Q*}\cE}$ is an overholonomic 
module over $\bbY$. In this case, all complexes are single modules in degree zero. The top horizontal arrow of the diagram
is induced by the following map : $$ \xymatrix@R=0pc {  \cE  \ar@{->}[r] & \cH om_{\sD_{Y_{\Q}}} (\cH
om_{\sD_{Y_{\Q}}}(\cE,\sD_{Y_{\Q}})) \ar@{->}[r] & \bbD_{Y_{\Q}}\circ \bbD_{Y_{\Q}}(\cE ) \\
   x \ar@{|->}[r] & ev_x(\varphi)=\varphi(x). & }$$
Recall that $C$ : $\cF \ra \bbD_{\bbY}\bbD_{\bbY}\cF$ is defined as follows in our case : as $\cF$ is overholonomic 
over $\bbY$, we have the identifications
$$\cF\simeq
(^{\dagger}Z_s)\cF=\sD^{\dagger}_{\cP}(^{\dagger}Z_s)\ot_{\sD^{\dagger}_{\cP}}\cF.$$
We therefore deduce, using the base change result \cite[1.4,4.4]{Virrion-dual}, that 
\begin{align*} \bbD_{\bbY}(\cF) &= \sD^{\dagger}_{\cP}(^{\dagger}Z_s)\ot_{\sD^{\dagger}_{\cP}}\bbD_{\bbP}(\cF) \\
                                &= R \cH om_{\sD^{\dagger}_{\cP}(^{\dagger}Z_s)}(\cF,\sD^{\dagger}_{\cP}(^{\dagger}Z_s)[d])
                               \ot_{\cO_{\cP}}\omega_{\cP},\\
                                &= R \cH om_{\sD^{\dagger}_{\cP}(^{\dagger}Z_s)}(\cF,\sD^{\dagger}_{\cP}(^{\dagger}Z_s)[d])
                            \ot_{\cO_{\cP}(^{\dagger}Z_s)}\omega_{\cP}(^{\dagger}Z_s),
\end{align*}
and the canonical map $C$ is then given by the following composition 
$$ \xymatrix{\cF  \ar@{->}[r] & \cH om_{\sD^{\dagger}_{\cP}(^{\dagger}Z_s)} (\cH
om_{\sD^{\dagger}_{\cP}(^{\dagger}Z_s)}(\cF,\sD^{\dagger}_{\cP}(^{\dagger}Z_s))) \ar@{->}[r] & 
\bbD_{\bbY}\circ \bbD_{\bbY}(\cF )}.$$
Note that $C$ is an isomorphism since it is an isomorphism when restricted to $\cP\setminus Z $ by ~\cite[4.3.10]{BerthelotDI}.
Moreover one has a canonical isomorphism
$$\cF \simeq \sD^{\dagger}_{\cP}(^{\dagger}Z_s)\ot_{\osD_{P_{\Q}}(*Z_{\Q})}\overline{j_{\Q+}\cE},$$
so that we can use again~\cite[1.4,4.4]{Virrion-dual}, to get the following identification
$$ \bbD_{\bbY}(\cF)\simeq \sD^{\dagger}_{\cP}(^{\dagger}Z_s)\ot_{\osD_{P_{\Q}}(*Z_{\Q})}
R \cH
om_{\osD_{P_{\Q}}(*Z_{\Q})}(\overline{j_{\Q+}\cE},\osD_{P_{\Q}}(*Z_{\Q})[d])
\ot_{\ocO_{P}}\overline{\omega}_{P_{\Q}}(*Z_{\Q}),$$
$$ \bbD_{\bbY}(\cF)\simeq \sD^{\dagger}_{\cP}(^{\dagger}Z_s)\ot_{\osD_{P_{\Q}}(*Z_{\Q})}
  \overline{j_{\Q+}\bbD_{Y_{\Q}}\cE}.$$
Here, we identify
 $${j_{\Q+}\bbD_{Y_{\Q}}\cE}=R \cH om_{\sD_{P_{\Q}}(*Z_{\Q})}({j_{\Q+}\cE},\sD_{P_{\Q}}(*Z_{\Q})[d])
\ot_{\cO_{P}}{\omega}_{P_{\Q}}(*Z_{\Q}).$$ 
Using again ~\cite[1.4,4.4]{Virrion-dual} applied to the sheaves $\osD_{P_{\Q}}(*Z_{\Q})$ and 
$\sD^{\dagger}_{\cP}(^{\dagger}Z_s)$, we find a canonical isomorphism 
$$ \bbD_{\bbY}\bbD_{\bbY}(\cF)\simeq \sD^{\dagger}_{\cP}(^{\dagger}Z_s)\ot_{\osD_{P_{\Q}}(*Z_{\Q})}
 \overline{ j_{\Q+}\bbD_{Y_{\Q}}\bbD_{Y_{\Q}}\cE},$$ that allows us to write the diagram of the statement in the 
following way
$$\xymatrix{ \overline{j_{\Q+}\cE} \ar@{->}[r]^>>>>>>>>>>>>>>>>>{\overline{j_{\Q+} c_{\Q}}} \ar@{->}[d]^{1 \ot id_{j_{\Q+}\cE}} &
\overline{j_{\Q +}\bbD_{Y_{\Q}}\circ\bbD_{Y_{\Q}}(\cE)}\ar@{->}[d]^{1 \ot id}_{} \\
 J_+(\sD^{\dagger}_{\cP}(^{\dagger}Z_s)\ot_{\osD_{P_{\bbQ}}(*Z_s)}\overline{j_{\Q +}\cE}\ar@{->}[r]^>>>>>{J_+C}) &
 \sD^{\dagger}_{\cP}(^{\dagger}Z_s)\ot_{\osD_{P_{\Q}}(*Z_{\Q})}\overline{j_{\Q+}\bbD_{Y_{\Q}}\bbD_{Y_{\Q}}\cE}.  
 }$$

The commutativity of this diagram comes then from the fact that if $x$ is a local section of 
$\overline{j_{\Q+}\cE}$, we have $(1 \ot id)(ev_x)=ev(1 \ot x)$.
\end{proof}
Under the same hypothesis as in the previous proposition ($Z$ is a transversal divisor of $P$) and with the same 
notations, we have the
\begin{cor} \label{bidual2} Let $\cE\in D^b_{hol}({Y_{\bbQ}})$ and suppose that 
$\cF=\sD^{\dagger}_{\cP}\ot_{\osD_{P_{\bbQ}}}\overline{j_{\Q +}\cE}\in D^b_{ovhol}({\bbY}).$
There is a commutative diagram
$$\xymatrix{ \overline{j_{\Q+}} \cE  \ar@{->}[r]^{\sim}\ar@{->}[d] &\overline{j_{\Q+}j_{\Q}^!j_{\Q!} \cE} \ar@{->}[d]\\
   \sD^{\dagger}_{\cP}\ot_{\osD_{P_{\bbQ}}}\overline{ j_{\Q+}\cE }\ar@{->}[r]^<<<<{\sim} & 
     J_+J^!J_! (\sD^{\dagger}_{\cP}\ot_{\osD_{P_{\bbQ}}}\overline{j_{\Q +}\cE})
. }$$
\end{cor}
\begin{proof} We have the following equality as functors on $D^b_{hol}({Y_{\bbQ}})$
\begin{align*} j_{\Q}^! j_{\Q!} & = j_{\Q}^!\bbD_{P_{\Q}}j_{\Q+} \bbD_{Y_{\Q}} \\
                                   & = \bbD_{Y_{\Q}} j_{\Q}^!j_{\Q+} \bbD_{Y_{\Q}} \\
                                    & = \bbD_{Y_{\Q}}\bbD_{Y_{\Q}}\simeq {\rm id}.
\end{align*}
On the other hand, let us notice that $J^!=\sD^{\dagger}_{\cP}(^{\dagger}Z_s)\ot_{\sD^{\dagger}_{\cP}}(-)$ is a scalar 
extension, so that again by \cite[1.4,4.4]{Virrion-dual}, $J^!\bbD_{\bbP}=\bbD_{\bbY}J^!$. Moreover 
for $\cF\in D^b_{\rm ovhol}(\bbY)$ one has $$J^!J_+ \cF=\sD^{\dagger}_{\cP}(^{\dagger}Z_s)\ot_{\sD^{\dagger}_{\cP}} \cF \simeq
\cF,$$ essentially by definition of the objects of $D^b_{\rm ovhol}(\bbY)$. Using these remarks, we compute
\begin{align*} J^! J_! & = J^! \bbD_{\bbP} J_+ \bbD_{\bbY} \\
                       & =   \bbD_{\bbY} J^! J_+ \bbD_{\bbY} \\
                       & = \bbD_{\bbY} \bbD_{\bbY}\simeq {\rm id},
\end{align*}
so that the diagram of the corollary is the same as the diagram of the previous proposition~\ref{bidual}.
\end{proof}

We next give another compatibility statement.
\begin{prop} \label{adjoint} Let $\cE\in D^b_{hol}({P_{\bbQ}})$ and suppose that
$\cF=\sD^{\dagger}_{\cP}\ot_{\osD_{P_{\bbQ}}}\overline{j_{\Q +}\cE}\in D^b_{ovhol}({\bbP}).$ 
Let $can: \cE \ra j_{\Q+}j_{\Q}^! \cE$ and 
$CAN: \cF \ra J_+J^! \cF$ be the canonical morphisms. Then the following diagram is commutative
$$\xymatrix{ \cE   \ar@{->}[r]^{can} \ar@{->}[d] & j_{\Q+}j_{\Q}^! \cE \ar@{->}[d]\\
                   \sD^{\dagger}_{\cP}\ot_{\osD_{P_{\bbQ}}}\cE\ar@{->}[r]^>>>>>{CAN} & J_+J^!
(\sD^{\dagger}_{\cP}\ot_{\osD_{P_{\bbQ}}}\cE).
}$$
\end{prop}
\begin{proof} Writing out explicitly all functors, we see that the above diagram comes down to the commutative diagram
$$\xymatrix{ \cE   \ar@{->}[r]^>>>>>>>>{can} \ar@{->}[d] &  \osD_{P_{\bbQ}}(*Z_{\bbQ})\ot_{\osD_{P_{\bbQ}}}\cE \ar@{->}[d]\\
                   \sD^{\dagger}_{\cP}\ot_{\osD_{P_{\bbQ}}}\cE\ar@{->}[r]^>>>>>{CAN} &
                   \sD^{\dagger}_{\cP}(^{\dagger}Z_s)\ot_{\osD_{P_{\bbQ}}}\cE.
}$$
\end{proof}

We recall the following result of Berthelot \cite[4.3.2]{Berthelot_Trento}. Recall that a relative normal crossing divisor is transversal.
\begin{prop}\label{rap_trento} Let $Z\subset P$ be a relative normal crossing divisor. Then one has
$$\cO_{\cP,\Q}({^\dagger Z_s})\simeq\sD^{\dagger}_{\cP}\ot_{\osD_{P_{{\bbQ}}}}
\overline{j_{\Q *}\cO_{Y_{\bbQ}}}.$$
\end{prop} 
Note that the sheaf $\overline{j_*\cO_{Y_{\bbQ}}}$ is equal to $\overline{\cO_{P_{\bbQ}}(*Z_{\Q})}$ and the 
isomorphism is given by the canonical inclusion of sheaves of rings 
$\overline{\cO_{P_{\bbQ}}(*Z_{\Q})} \hrig \cO_{\cP,\Q}(^{\dagger}Z_s)$, 
sending $1$ to $1$. This allows us to identify $\cO_{\cP,\Q}({^\dagger Z_s})$ with 
$\sD^{\dagger}_{\cP}\ot_{\osD_{P_{{\bbQ}}}}\overline{j_*\cO_{Y_{\bbQ}}}=\cO_{\bbY}$.
In the same situation as in~\ref{bidual} we have
\begin{prop}\label{Jexcl} Let $Z\subset P$ be a relative normal crossing divisor.

\begin{enumerate} \item $\bbD_{Y_{\Q}}\cO_{Y_{\Q}}=\cO_{Y_{\Q}}, \; 
                           \bbD_{\bbY}\cO_{\bbY}=\cO_{\bbY},$
                     \item there is a canonical isomorphism $J_!\cO_{\bbY}\simeq 
                    \sD^{\dagger}_{\cP}\ot_{\osD_{P_{\bbQ}}}j_{\Q!}\cO_{Y_{\Q}}$.
             \end{enumerate}
\end{prop}
\begin{proof} The fact that $\bbD_{Y_{\Q}}\cO_{Y_{\Q}}=\cO_{Y_{\Q}}$ is classical and comes from the fact 
that on the smooth scheme $Y_{\Q}$, the $\sD_{Y_{\Q}}$-module $\cO_{Y_{\Q}}$ admits a resolution by the Spencer complex, and that this latter 
complex is auto-dual. To see the second statement, we use the fact proved in~\cite[Lemme 4.2.1]{HuFour} 
that $\cO_{\bbY}=\cO_{\cP,\Q}(^{\dagger}Z_s)$ also admits a 
resolution by a Spencer complex (with $d=\dim\,Y_{\Q}$)
$$0 \ra \sD^{\dagger}_{\cP}(^{\dagger}Z_s)\ot_{\cO_{\cP}}\Lam^d \cT_{\cP} \ra \ldots \ra 
 \sD^{\dagger}_{\cP}(^{\dagger}Z_s)\ot_{\cO_{\cP}}\Lam^1 \cT_{\cP}\ra  \sD^{\dagger}_{\cP}(^{\dagger}Z_s)$$
that is auto-dual for the functor $\bbD_{\bbY}=
R\cH om_{\sD^{\dagger}_{\cP}(^{\dagger}Z_s)}(\, \cdot \,,\sD^{\dagger}_{\cP}(^{\dagger}Z_s) [d])\ot_{\cO_{\cP}}\omega_{\cP}$.
 This proves (i). Then we compute 
\begin{align*} J_!\cO_{\bbY} & = \bbD_{\bbP} J_+ \bbD_{\bbY}\cO_{\bbY} \\
                              & = \bbD_{\bbP} J_+ \cO_{\bbY} \\
                             & = \bbD_{\bbP}(\sD^{\dagger}_{\cP}\ot_{\osD_{P_{\bbQ}}}\overline{j_+\cO_{Y_{\Q}}}) 
                           \textrm{ by ~\ref{rap_trento}},\\
                             & \simeq \sD^{\dagger}_{\cP}\ot_{\osD_{P_{\bbQ}}}\overline{\bbD_{P_{\Q}} (j_+\cO_{Y_{\Q}})}
 \textrm{ by ~\ref{compat_dual},}
\end{align*}
which proves (ii).
\end{proof}

\vskip5pt

\subsection{Proper morphisms} \label{debcompat2}

Before giving compatibility results for direct images relative to proper morphisms we establish the following auxiliary lemmas.
\begin{lemma}\label{mapbar} Let $f$ : $ P \ra Q$ be a morphism of smooth $\fro$-schemes and $\cF\in D^+_{qcoh}(\cO_{P_{\bbQ}})$ with  $\ocF\in D^+(\ocO_{P_{\bbQ}})$. There is a natural map $\overline{Rf_{\bbQ *}(\cF)}\ra R\hf_*(\ocF)$. 
\end{lemma}
\begin{proof} We have the following diagram 
$$\xymatrix{ \cP\ar@{->}[r]^{\alpha} \ar@{->}[d]^{\hf} & P \ar@{->}[d]^{f}& P_{\bbQ}\ar@{->}[l]^{j}
\ar@{->}[d]^{f_{\bbQ}}  \\  
\cQ\ar@{->}[r]^{\alpha} & Q & Q_{\bbQ}\ar@{->}[l]^{j}
, }$$
in which both squares are commutative diagrams (the left one is commutative as it is commutative when $\cP$ and 
$\cQ$ are replaced by $P_i$ and $Q_i$). Let $\cE$ be a quasi-coherent sheaf on $P_{\Q}$, we have a canonical map $j_*\cE\ra
\alpha_*\alpha^{-1}j_*\cE$. If we compose this map with $f_*$, we get by adjunction by $\alpha$ 
a map $\overline{f_{\bbQ *}\cE}\ra
\hf_*\overline{\cE}$.
Let $\cF\stackrel{\simeq}{\ra}\cI_{\bul,\bul}$ be an injective resolution of $\cF$ by a double complex of quasi-coherent 
$\cO_{P_{\bbQ}}$-modules. As the functor
$\cE \mapsto \ocE$ is exact on quasi-coherent $\cO_{P_{\bbQ}}$-modules, we have a quasi-isomorphism
$ R\hf_*(\overline{\cF})\simeq R\hf_*(\overline{\cI}_{\bul,\bul}).$  
 We finally obtain the map of the lemma by the following
composition
$$ \overline{Rf_{\bbQ *}(\cF)} \simeq  \overline{f_{\bbQ *}(\cI_{\bul,\bul})} \ra \hf_*(\overline{\cI}_{\bul,\bul})\ra 
R\hf_*(\overline{\cI}_{\bul,\bul})\simeq R\hf_*(\overline{\cF}) .$$
\end{proof}

\begin{lemma}\label{map_im_dir} Let $f$ : $P \ra Q$ be a morphism of smooth $\fro$-schemes and $\cE\in D^b_{coh}(\sD_{P_{\bbQ}})$. Then 
there is a canonical morphism in $D^+(\osD_{Q_{\bbQ}})$
$$\sD^{\dagger}_{\cQ}\ot_{\osD_{Q_{\bbQ}}}\overline{f_{\bbQ +}(\cE)}\ra 
          \hf_+\left(\sD^{\dagger}_{\cP}\ot_{\osD_{P_{\bbQ}}}\ocE\right).$$
\end{lemma}
\begin{proof} It is enough to prove that there is a map
          $$  \overline{f_{\bbQ +}(\cE)}\ra 
          \hf_+\left(\sD^{\dagger}_{\cP}\ot_{\osD_{P_{\bbQ}}}\ocE\right).$$
Let us introduce the transfer sheaves 
$\sD_{Q_{\bbQ}\la P_{\bbQ}}=\omega_{P/Q}\otimes_{\cO_{P_{\Q}}}f_{\bbQ}^*\sD_{Q_{\bbQ}}$, and 
$$\hsD^{(m)}_{\cP \ra \cQ}=\varprojlim_i f_i^* \sD^{(m)}_{Q_i}, \hskip10pt
\hsD^{(m)}_{\cQ \la \cP}=\omega_{\cP/\cQ}\ot_{\cO_{\cP}}\hsD^{(m)}_{\cP \ra \cQ} ,\hskip10pt
\sD^{\dagger}_{\cQ \la \cP}=\varinjlim_m \hsD^{(m)}_{\cQ \la \cP,\Q}.$$
Recall that 
$$ f_{\bbQ +}(\cE)= Rf_{\bbQ *}\left(\sD_{Q_{\bbQ}\la P_{\bbQ}}\ot_{\sD_{P_{\bbQ}}}^{\bL}\cE\right), \hskip10pt
\hf_+\left(\sD^{\dagger}_{\cP}\ot_{\osD_{P_{\bbQ}}}\ocE\right)=R\hf_*\left(\sD^{\dagger}_{\cQ \la \cP}
\ot_{\osD_{P_{\bbQ}}}^{\bL}\ocE\right).$$
Note that we have
$$ \overline{f_{\Q}^{-1}\sD}_{Q_{\bbQ}}\simeq \alpha^{-1}f^{-1}(\sD_{Q_{}}\ot \Q) = \hf^{-1}\alpha^{-1}(\sD_{Q_{}}\ot \Q),$$
and for all $m$ we have maps
$\alpha^{-1}(\sD_{Q_{}})\ra  \hsD_{Q_{}}^{(m)} $. This gives maps 
$\hf^*(\alpha^{-1}(\sD_{Q_{}}))\ra \hsD^{(m)}_{\cP \ra \cQ}$, that give rise to maps of transfer sheaves 
$\osD_{Q_{\Q}\la P_{\Q}} \ra \sD^{\dagger}_{\cQ \la \cP}$.

Take $\cE\in D^b_{coh}(\sD_{P_{\bbQ}})$, then, observing that $\sD_{Q_{\bbQ}\la P_{\bbQ}}$ is a quasi-coherent
$\cO_{P_{\Q}}$-module, we see that
$$   \sD_{Q_{\bbQ}\la P_{\bbQ}}\ot_{\sD_{P_{\bbQ}}}^{\bL}\cE \in  D^b_{qcoh}(\cO_{P_{\Q}}), $$ so that we can apply
~\ref{mapbar} to this complex of sheaves and we finally get the map of the statement by the following composition
$$ \overline{Rf_{\bbQ *}\left(\sD_{Q_{\bbQ}\la P_{\bbQ}}\ot_{\sD_{P_{\bbQ}}}^{\bL}\cE\right)} \ra 
  R\hf_*\left(\overline{\sD_{Q_{\bbQ}\la P_{\bbQ}}\ot_{\sD_{P_{\bbQ}}}^{\bL}\cE}\right) \ra 
   R\hf_*\left( \sD^{\dagger}_{\cQ \la \cP} \ot_{\osD_{P_{\bbQ}}}^{\bL}\ocE\right).$$ 
\end{proof}

\vskip5pt

Assume from now on and for the rest for this subsection that $f: P \ra Q$ is a proper map of smooth $\fro$-schemes. Both
sheaves $\sD_{P_{\bbQ}}$ and $\sD^{\dagger}_{\cP}$ have finite cohomological dimension ~\cite[4.4.8]{Berthelot_Dmod2}, as well as $Rf_*$
since $f$ is proper. Let $*\in \{-,b\}$ and $\cE\in D^*_{coh}(\sD_{P_{\bbQ}})$, then
 $ f_{\bbQ +}(\cE)$ (resp.
$\hf_+(\sD^{\dagger}_{\cP}\ot_{\osD_{P_{\bbQ}}}\ocE)$ are objects of  $D^*_{coh}(\osD_{Q_{\bbQ}}) $, 
(resp. $D^*_{coh}(\sD^{\dagger}_{\cQ})$), and thanks to the lemma, there is a map 
$$ \sD^{\dagger}_{\cQ}\ot_{\osD_{Q_{\bbQ}}}\overline{f_{\bbQ +}(\cE)}\ra 
         \hf_+(\sD^{\dagger}_{\cP}\ot_{\osD_{P_{\bbQ}}}\ocE).$$
Our goal (see ~\ref{compat_proj}) is to prove that this map is an isomorphism when $P$ and $Q$ are projective $\fro$-schemes. As
usual, we will factorize $f$ as a closed immersion and a projection. We first deal with the case of a closed immersion.

\vskip5pt

Let $i$ : $P \hra Q$ be a closed immersion of smooth  $\fro$-schemes, defined by a sheaf of ideals $\cI \subset \cO_{Q}$. 
  We then have the following compatibility result for closed immersions.
\begin{prop} \label{compat-im}Let $*\in\{-,b\}$ and let $\cE \in D^*_{coh}(\sD_{P_{\bbQ}})$. Then there is a canonical isomorphism
        in $D^*_{coh}(\sD^{\dagger}_{Q})$
       $$\sD^{\dagger}_{\cQ}\ot_{\osD_{Q_{\bbQ}}}\overline{i_{\bbQ +}(\cE)}\simeq 
\hi_+\left(\sD^{\dagger}_{\cP}\ot_{\osD_{P_{\bbQ}}}\ocE \right).$$
\end{prop}
\begin{proof} It is well known that $i_{+\Q}$ sends $D^*_{coh}(\sD_{P_{\bbQ}})$ to $D^*_{coh}(\sD_{Q_{\bbQ}})$, resp. 
$\hi_+$ sends $D^*_{coh}(\sD^{\dagger}_{P})$ to $D^*_{coh}(\sD^{\dagger}_{Q})$. Finally both functors send
$D^*_{coh}(\sD_{P_{\bbQ}})$ to $D^*_{coh}(\sD^{\dagger}_{Q})$.
The map from the left-hand side to the right-hand side is the one given in the previous lemma ~\ref{map_im_dir}.
              Since $i$ is a closed immersion, $i$ is affine, it has finite cohomological dimension and 
both functors are way out left in the sense of \cite[I,7]{hartshorne_res_dual}. Proving that the map is an isomorphism
is a local question on $Q$, so that we can assume that $P$ and $Q$ are affine. In this case any coherent $\sD_{P_{\bbQ}}$-module is a
quotient of a finite free $\sD_{P_{\bbQ}}$-module, and using a standard dévissage argument for way out left functors
~\cite[I,7, (iv)]{hartshorne_res_dual} we are reduced to prove the statelent in the case where $\cE=\sD_{P_{\bbQ}}$.
 In this case, we have the following formulas
\begin{align*}i_{\bbQ +}(\sD_{P_{\bbQ}})=i_*\left(\sD_{Q_{\bbQ}\la P_{\bbQ}}\right),\hskip10pt
               \hi_+\left(\sD^{\dagger}_{\cP}\right)  = \hi_*\left(\sD_{\cQ \la \cP}^{\dagger}\right).\end{align*}
As $\hi$ is a quasi-compact morphism, $R\hi_*=\hi_*$ commutes with inductive limits so that 
           $$     \hi_*\left(\sD^{\dagger}_{\cQ \hla \cP}\right)  = 
\varinjlim_m \hi_*\left(\hsD_{\cQ \hla\cP,\bbQ}^{(m)}\right). $$
 Let us fix an integer $m$.
 We have to show that 
\begin{gather}\label{keyf1}
 \hi_*\left(\hsD_{\cQ \hla\cP,\bbQ}^{(m)}\right)\simeq
\hsD_{\cQ,\bbQ}^{(m)}\ot_{\osD_{P_{\bbQ}}}\overline{i_*\left(\sD_{Q_{\bbQ}\la P_{\bbQ}}\right)}.\end{gather}
We first compute the left hand side of this formula. 
By \cite[Thm.3.5.3]{BerthelotIntro}, we know that $\hi_+(\hsD_{\cP}^{(m)})$ is a
coherent $\hsD_{\cQ}^{(m)}$-module, and as $\hi$ is affine, we have by ~\cite[prop. 13.2.3]{EGA3_I}
\begin{gather} \label{lhs} \addtocounter{para}{1}
\hi_+(\hsD_{\cP}^{(m)})\simeq \varprojlim_i 
i_* (\sD^{(m)}_{Q_i \hla P_i}).
\end{gather}
We now come to the right hand side of the formula ~\ref{keyf1}.
Let us first recall the following 
\begin{lemma} The sheaf $i_+(\sD^{(m)}_{P})$ is a coherent $\sD^{(m)}_{Q}$-module.
\end{lemma}
\begin{proof} We have 
\begin{align*} i_+\sD^{(m)}_{P} & =  i_*\left(i^*\omega_Q^{-1}\ot_{\cO_Q} i^*\sD^{(m)}_{Q}\ot_{\cO_P}\omega_P\right)\\
                                & \simeq   \omega_Q^{-1}\ot_{\cO_Q} \sD^{(m)}_{Q}\ot_{\cO_Q}i_*\omega_P \textrm{  by
the projection formula},
\end{align*}
the left $\sD^{(m)}_Q$-module structure being given by the one of $\omega_Q^{-1}\ot_{\cO_Q} \sD^{(m)}_{Q}$, that is by
the right structure on $\sD^{(m)}_{Q}$ twisted on the left, which makes this left $\sD^{(m)}_{Q}$-module a coherent
module.
\end{proof}

Consider now the following $\hsD^{(m)}_{\cQ}$-module, that is coherent by the previous lemma,
$$\cM=\hsD_{\cQ}^{(m)}\ot_{\alpha^{-1}\sD^{(m)}_{Q}}\alpha^{-1}  i_+\sD^{(m)}_{P}.$$
As $\cM$ is coherent, by ~\cite[3.2.4]{BerthelotDI}, we have
\begin{align*}
            \cM & \simeq \varprojlim_i \sD^{(m)}_{Q_i}\ot_{\sD^{(m)}_{Q_i}}i_* \sD^{(m)}_{Q_i\hla
P_i}  \\
            & \simeq \varprojlim_i i_* \sD^{(m)}_{Q_i \hla P_i}.
\end{align*}
As $\cM_{\bbQ}$ coincides with $\hsD_{\cQ}^{(m)}\ot_{\osD_{Q_{\bbQ}}}
\overline{i_{\Q+}(\sD_{P_{\Q}})}$, this module is isomorphic 
with the right-hand side of ~\ref{keyf1}.
We finally obtain our result by comparing with the left-hand side of ~\ref{keyf1}, which is computed using 
~\ref{lhs}, and gives exactly the same thing.
\end{proof}

\vskip5pt

As before, let $*\in\{b,-\}$.
\begin{prop} \label{compat_proj} Let $P$, $Q$ be smooth and projective $\fro$-schemes, let $f: P \ra Q$ be a proper
morphism with formal completion $\hf: \cP \ra \cQ$.
For any $\cE\in D^*_{coh}(\sD_{P_{\Q}})$, there is a natural isomorphism in
$D^*_{coh}(\sD^{\dagger}_{\cQ})$
$$\sD^{\dagger}_{\cQ}\ot_{\osD_{Q_{\bbQ}}}\overline{f_{\bbQ +}(\cE)}\simeq \hf _+\left(\sD^{\dagger}_{\cP}\ot_{\osD_{P_{\bbQ}}}\ocE\right) .$$
\end{prop}
\begin{proof}

We already noticed that both functors send objects of $D^*_{coh}(\sD_{P_{\Q}})$ to objects of
 $D^*_{coh}(\sD^{\dagger}_{\cQ})$, as $f$ has finite cohomological dimension. Moreover both functors are way out left 
in the sense of \cite[I,7]{hartshorne_res_dual}.
The map from left-hand side to the right-hand side was defined in 
~\ref{map_im_dir}. 
 Using 
~\cite[\href{https://stacks.math.columbia.edu/tag/0C4Q}{Tag 0C4Q}]{StacksP}, we see that the morphism $f$ is projective.
Then, using the previous compatibility result~\ref{compat-im} for closed immersions, it is enough to prove the statement 
when $P$ is a relative
projective space over $Q$, say $P=\bbP^M_Q$ and $f$ : $\bbP^M_Q \ra Q$ is the canonical map.
 Since the question is local on $Q$, we can (and we do) assume that $Q$
 is affine and smooth with coordinates $t_1,\ldots,t_s$. Let $\cE$ be a coherent $\sD_{P,\Q}$-module. 
As $P_{\bbQ}$ is a noetherian space, $\cE$ is an inductive limit of its sub $\cO_{P_{\bbQ}}$-coherent sheaves, so that
there is a $\cO_{P_{\bbQ}}$-coherent sheaf $\cE'$ and a surjection of $\sD_{P_{\bbQ}}$-modules 
$\sD_{P_{\bbQ}}\ot_{\cO_{P_{\bbQ}}}\cE'\thra \cE$, where the $\sD_{P_{\bbQ}}$-module structure on the left hand side is
given by the one of $\sD_{P_{\bbQ}}$.
By Serre's theorem, for some $a,r\in \bbN$, there is a surjection of coherent $\cO_{P_{\bbQ}}$-modules 
$\cO_{P_{\bbQ}}(-a)^r \thra \cE'$, and we see that there is a surjection of coherent $\sD_{P_{\bbQ}}$-modules 
 $\sD_{P_{\bbQ}}(-a)^r \thra \cE$. Iterating this process, we see that each coherent $\sD_{P_{\bbQ}}$-module has some
resolution by $\sD_{P_{\bbQ}}$-modules of the type $\sD_{P_{\bbQ}}(-a)^r$. Finally using again the dévissage argument 
for way out left functors of ~\cite[I,7, (iv)]{hartshorne_res_dual} we are reduced to prove the proposition for 
a projective morphism $f: P=\bbP^M_Q \ra Q$, with $Q$ affine, endowed with coordinates, and $\cE=\sD_{P_{\bbQ}}(-a)$,
with $a\in \bbN$. Let us assume this from now on. Since the functor $Rf_*$ commutes with inductive limits, because $\bbP^M_Q$ and $Q$ are quasi-compact,
  it is also enough to prove
     that, for all $m$, we have 
\begin{gather} \label{fixm}
        \hsD^{(m)}_{\cQ}\ot_{\osD_{Q_{\bbQ}}}\overline{f_+(\sD_{P_{\Q}}(-a))}\simeq 
\hf_+\left(\hsD^{(m)}_{\cP,\bbQ}\ot_{\osD_{P_{\bbQ}}}\osD_{P_{\Q}}(-a)\right) .
\end{gather}

The following lemma therefore completes the proof of the proposition. 
\end{proof}
\begin{lemma} Assertion ~\ref{fixm} is true for any $m$.
\end{lemma}

\begin{proof}
Let $\cF= \sD^{(m)}_{P}(-a)$, we have 
\begin{align*}
f_+(\cF) & = Rf_*\left(\sD^{(m)}_{Q \la P}\ot_{\sD^{(m)}_P}\sD^{(m)}_P(-a) \right) \\
          & =  Rf_*\left(f^*\sD^{(m)}_{Q}\ot_{\cO_P}\omega_{P/Q}(-a)\right),\\
          & =  (\sD^{(m)}_{Q} \ot_{\cO_Q} \omega_Q^{-1})\ot_{\cO_Q} Rf_*(\omega_{P}(-a)) \textrm { (by the projection
formula)},
\end{align*}
where the left $\sD_Q^{(m)}$-module structure is given by the left structure of $\sD^{(m)}_{Q}\ot_{\cO_Q}
\omega_Q^{-1}$, obtained by twisting the right structure of $\sD^{(m)}_{Q}$. As $\omega_Q$ is free of rank $1$, 
$$\omega_P\simeq \omega_{P}\ot_{\cO_P}f^* \omega_Q^{-1}\simeq \omega_{P/Q} \simeq \cO_P(-M-1),$$
where $M:=\dim P_{\Q}-\dim Q_{\Q}$. We refer for example to ~\cite[III,thm. 5.1]{HartshorneA} for the computation 
of $Rf_*(\cO_P(-M-1))$ over any affine base $Q$, which is a complex of finite free $\cO_Q$-modules. More precisely,
denote $$d=\max\{rank(H^0(P,\cO_P(-a-M-1))),rank(H^M(P,\cO_P(-a-M-1)))\}.$$ 
There are several cases:
         \begin{enumerate}\item If $a\leq -M-1$, then $f_+(\cF)\simeq\sD^{(m)}_{Q}\ot_{\cO_Q}\cO_Q^d$ is concentrated in
degree $0$,
		\item if $a\geq 0$, then $f_+(\cF)\simeq \sD^{(m)}_{Q}\ot_{\cO_Q}\cO_Q^d[-M]$ is concentrated in degree $M$,
		\item if $-M \leq a \leq -1$, then $f_+(\cF)$=0.
		\end{enumerate}
Note also that we have the following isomorphism of (twisted) left 
$\hsD^{(m)}_{\cQ,\bbQ}$-modules
$$   \hsD^{(m)}_{\cQ,\bbQ}\ot_{\osD^{(m)}_{Q}}(\osD^{(m)}_{Q}\ot_{\ocO_{Q_{\bbQ}}} \overline{\omega_{Q}}^{-1})
\simeq \hsD^{(m)}_{\cQ,\bbQ}\ot_{\cO_{\cQ}}\omega_{\cQ}^{-1}.$$
We will first compute the left-hand side of ~\ref{fixm}. Let us denote 
\begin{align*}
\cA &:=  \hsD^{(m)}_{\cQ}\ot_{\alpha^{-1}\sD^{(m)}_{Q}}\alpha^{-1}{f_+(\cF)}\in D^b_{coh}(\hsD^{(m)}_{\cQ}) \\
     &=  (\hsD^{(m)}_{\cQ}\ot_{\cO_Q}\omega_{\cQ}^{-1})\ot_{\alpha^{-1}\cO_{Q}}\alpha^{-1}{Rf_*(\omega_{P}(-a))}.
\end{align*}
This is a complex concentrated in at most one degree, where
 it is isomorphic to a direct sum of $d$ copies of $\hsD^{(m)}_{\cQ}$. In particular, by 
\cite[3.2.1]{BerthelotIntro}, it satisfies $\cA\simeq R\varprojlim_i (\sD^{(m)}_{Q_i}\ot^{L}_{\hsD^{(m)}_{\cQ}}\cA)$. Since 
$\cA$ is a complex of finite free $\hsD^{(m)}_{\cQ}$-modules, we have that 
$$ \sD^{(m)}_{Q_i}\ot^{L}_{\hsD^{(m)}_{\cQ}}\cA=\sD^{(m)}_{Q_i}\ot_{\hsD^{(m)}_{\cQ}}\cA 
\simeq \sD^{(m)}_{Q_i}\ot_{\cO_{Q_i}}\omega^{-1}_{{Q_i}}
\ot_{\alpha^{-1}\cO_Q}\alpha^{-1}{Rf_*(\omega_{P}(-a))},$$
 is a complex either in degree $M$ or $0$, where it is isomorphic to a direct sum of $d$ copies of $\sD^{(m)}_{Q_i}$. Finally 
we have 
$$\cA \simeq R\varprojlim_i \left(\sD^{(m)}_{Q_i}\ot_{\cO_{Q_i}}\omega^{-1}_{{Q_i}}
\ot_{\alpha^{-1}\cO_{Q}}\alpha^{-1}{Rf_*(\omega_{P}(-a))}\right).$$

To compute the right-hand side of ~\ref{fixm}, we introduce 
$\cB=\hf_+(\hsD^{(m)}_{\cP}(-a))$, so that we have 
\begin{align*} \cB & \simeq
R\hf_*\left(\hf^*(\hsD^{(m)}_{\cQ}\ot_{\cO_{\cQ}}\omega^{-1}_{\cO_{\cQ}})\ot_{\cO_{\cP}}\omega_{\cP}(-a)\right) \\
                 & \simeq  R\hf_*R\varprojlim_i\left(f_i^*(\sD^{(m)}_{Q_i}\ot_{\cO_{Q_i}}\omega^{-1}_{Q_i})
\ot_{\cO_{P_i}}\omega_{P_i}(-a)\right) \\
                 & \simeq  R\varprojlim_i Rf_{i*}\left(f_i^*(\sD^{(m)}_{Q_i}\ot_{\cO_{Q_i}}\omega^{-1}_{Q_i})
\ot_{\cO_{P_i}}\omega_{P_i}(-a)\right) 
\textrm{ by ~\cite[\href{https://stacks.math.columbia.edu/tag/0BKP}{Tag 0BKP}]{StacksP} }\\
                 & \simeq  R\varprojlim_i \left(\sD^{(m)}_{Q_i}\ot_{\cO_{Q_i}}\omega^{-1}_{Q_i}\ot_{\cO_{Q_i}}
                 Rf_{i*}\omega_{P_i}(-a)\right).
\end{align*}
Again, by using the computation of ~\cite[Thm. III.5.1]{HartshorneA}, we see that 
$$Rf_{i*}\omega_{P_i}(-a)\simeq Rf_{i*}\cO_{P_i}(-a-M-1)$$ is a complex
 concentrated in only one degree and $$Rf_{i*}\omega_{P_i}(-a)\simeq \cO_{Q_i}\ot_{\cO_{Q}} Rf_{*}\omega_{P}(-a). $$
This finally shows that $\cB$ is isomorphic to $\cA$. This implies the lemma.\end{proof}

\vskip5pt

\subsection{Compatibility for intermediate extensions of constant coefficients}\label{sec_compatibillity}
We now come to the main application of our previous compatibility results. For this we place ourselves in the following axiomatic situation (S): 
\vskip10pt
\begin{enumerate} \item[(i)] $Y$ is an affine and smooth scheme over $\fro$.\\
                   \item[(ii)] there is an immersion $v: Y \hrig P$ into a smooth projective scheme $P$ over $\fro$. Let $X:=\overline{Y}$ be
                    the Zariski closure of $Y$ in $P$ and $Z:=X \backslash Y$.\\
                   \item[(iii)] There is a smooth and projective $\fro$-scheme $X'$, a surjective morphism 
                   $$b: X' \ra X$$ 
                     inducing an isomorphism $Y':=b^{-1}Y \simeq Y$, such that $Z'=X' \backslash Y'$ is a transversal
                     divisor as defined in ~\ref{deftrans} with normal crossings. We have the open immersion 
                     $j': Y\simeq b^{-1}Y\hra X'$.\\
\end{enumerate}

As usual $\frX,\frY,...$ denote the formal schemes obtained from these schemes by $p$-adic completion, and 
$X_s,Y_s,...$ denote their special fiber.  For simplicity, we also write $v$ for the morphism of frames 

$$ v: \bbY=(Y_s,X_s,\cP) \longrightarrow (P_s,P_s,\cP)=\bbP$$

induced by the immersion $v: Y \hrig P$. Let us introduce the 
composite morphism $$g: X' \sta{b}{\ra}X \hra P.$$ By $p$-adic completion 
we obtain a morphism $\hg $: 
$\frX' \ra \cP$, and a morphism of frames $$u=(Id_{Y_s},b_s,\hg): \bbY'=(Y_s,X'_s,\frX')\ra \bbY=(Y_s,X_s,\cP).$$ Denoting 
$G=(g_s,g_s,\hg)$ and $J'=(j'_s,{\rm id}_{X'_s},{\rm id}_{\frX'})$,
we then have the basic commutative diagram of frames:
$$\xymatrix{ \bbY'=(Y_s,X'_s,\frX') \ar@{->}[r]^{J'}\ar@{->}[d]^{u} & (X'_s,X'_s,\frX')=\bbX'\ar@{->}[d]^{G} \\
               \bbY=(Y_s,X_s,\cP) \ar@{->}[r]^{v} & (P_s,P_s,\cP)=\bbP.
}$$
The frame morphism $u$ is c-affine, and the first morphism of this frame is equal to the identity,
 so that by ~\cite[1.2.8]{AbeCaro}, we know that 
$u_!$ and $u_+$ are $t$-exact, and that we have $u_!=u_+$, as functors of abelian categories $F$-$\text{Ovhol}(\bbY')\ra
F$-$\text{Ovhol}(\bbY)$, and that these functors are in fact equal to ${\cH^0_t u_+=\cH^0_t u_!}$.
Let $$Q:=v\circ u=G\circ J',$$ which is a c-affine immersion, (in particular $Y_s \hra P_{s}$ is an immersion). 
Note that we have $$ J'_+ \cO_{\bbY'}= \cO_{\frX',\Q}({^\dagger Z'_s}),$$ and that in our case 
$J'_+$ is the forgetful functor $F$-$\text{Ovhol}(\bbY')\ra F$-$\text{Ovhol}(\bbX')$.
Let $v_{\bbQ}$ be the immersion $Y_{\bbQ} \hra P_{\bbQ}$.
We now fix once and for all the following notations: 
$$
\begin{array}{ccccc}
 \theta_{v_{\Q}} &  :=  & \theta_{v_{\Q},\cO_{Y_{\bbQ}}}: {v_{\bbQ!}\cO_{Y_{\bbQ}}}\ra 
{v_{\bbQ +}\cO_{Y_{\bbQ}}}   &\hskip20pt  \text{resp.} \hskip20pt  & \theta_{j'_{\Q}}:=\theta_{j'_{\Q},\cO_{Y_{\bbQ}}}  \\
  &   &   & & \\
 \theta_Q:=\theta^0_{Q,\cO_{\bbY'}} & =  &  \theta_{Q,\cO_{\bbY'}}: Q_!\cO_{\bbY'}\ra Q_+\cO_{\bbY'}  & \hskip20pt \text{resp.} \hskip20pt
 & \theta_{J'}:=\theta_{J',\cO_{\bbY'}}. \\
   &   &   & & 
\end{array}
$$
We also need the two morphisms
$$  \theta_v^{alg}   :=   \overline{\theta}_{v_{\Q}}  \hskip20pt \text{resp.} \hskip20pt   \theta_{j'}^{alg}:=\overline{\theta}_{j'_{\Q}}.$$

Our goal is to describe the relation between the classical intermediate extension 
$v_{\bbQ!+}\cO_{Y_{\bbQ}}$ on the generic fibre and the Abe-Caro intermediate extension $v_{!+}\cO_{\bbY}$ on the special fibre. 

We start with the following lemma. 

\begin{lemma} \label{thetacomp0} We have the following commutative diagram in 
$F$-$\text{\rm Ovhol}(\bbX')$
$$\xymatrix{\sD^{\dagger}_{\frX'}\ot_{{\osD_{X'_{\Q}}}} \overline{j'_{\Q!}\cO_{Y_{\Q}}}
\ar@{->}[r]^{}
\ar@{->}[d]^{}_{(1)}^{\simeq}&
\sD^{\dagger}_{\frX'}\ot_{{\osD_{X'_{\Q}}}}\overline{j'_{\Q+}\cO_{Y_{\Q}}}   \ar@{->}[d]^{\simeq}_{(3)}    \\
  J'_!\cO_{\bbY'} \ar@{->}[r]^{\theta_{J'}} &     J'_+\cO_{\bbY'}
 }$$
where all maps are canonical and where the upper horizontal arrow equals $\sD^{\dagger}_{\frX'}\ot\theta_{j'}^{alg}$.
\end{lemma}
\begin{proof} The diagram of the statement can be completed by the following diagram
$$\xymatrix{\sD^{\dagger}_{\frX'}\ot_{{\osD_{X'_{\Q}}}} \overline{j'_{\Q!}\cO_{Y_{\Q}}}
\ar@{->}[r]^>>>>>{c_\Q}\ar@{->}@/^2pc/[rr]^{\sD^{\dagger}_{\frX'}\ot\theta_{j'}^{alg}}
\ar@{->}[d]^{}_{(1)}&
\sD^{\dagger}_{\frX'}\ot_{{\osD_{X'_{\Q}}}}\overline{j'_{\Q+}{j'}_{\Q}^{!} j'_{\Q!}\cO_{Y_{\Q}}}\ar@{->}[r]^>>>>>{\simeq}
\ar@{->}[d]^{}_{(2)} &
\sD^{\dagger}_{\frX'}\ot_{{\osD_{X'_{\Q}}}}\overline{j'_{\Q+}\cO_{Y_{\Q}}}   \ar@{->}[d]^{\simeq}_{(3)}    \\
  J'_!\cO_{\bbY'} \ar@{->}[r]^{C} \ar@{->}@/_2pc/[rr]^{\theta_{J'}} & J'_+{J'}^! J'_! \cO_{\bbY'} \ar@{->}[r]^{\simeq}&
    J'_+\cO_{\bbY'}
 .}$$ Let us prove that both squares of this diagram are commutative.
The isomorphism $(3)$ is given by Berthelot's result~\ref{rap_trento}. The right square of
 this diagram is commutative by ~\ref{bidual2}, and the horizontal maps of this square are isomorphisms, so
that (2) is an isomorphism as well. The left square of this diagram is commutative by ~\ref{adjoint} applied to 
$\overline{j'_{\Q!}\cO_{Y_{\Q}}}$ and ~\ref{Jexcl}. Moreover (ii) of ~\ref{Jexcl} tells us that (1) is an isomorphism.
We conclude that the external square is commutative with vertical arrows being isomorphisms.
\end{proof}
Recall that $Q=G\circ J'=v\circ u$, $v_{\Q}=g_{\Q}\circ {j'}_{\Q}$, 
 $\theta_{v}^{alg}=\overline{\theta}_{v_{\Q},\cO_{Y_{\Q}}}$ and 
            $\theta_{Q}=\theta_{Q,\cO_{\bbY}}$.
\begin{cor} \label{thetacomp} There is a commutative diagram (with canonical vertical maps) in 
           $F$-$\text{\rm Ovhol}(\bbP)$
           $$\xymatrix{\sD^{\dagger}_{\cP}\ot_{{\osD_{P_{\Q}}}} \overline{{v_{\Q!}}\cO_{Y_{\Q}}}
            \ar@{->}[r]^{}
             \ar@{->}[d]^{}^{\simeq}&
          \sD^{\dagger}_{\cP}\ot_{{\osD_{P_{\Q}}}}\overline{v_{\Q+}\cO_{Y_{\Q}}}  
                \ar@{->}[d]^{\simeq}    \\
             Q_!\cO_{\bbY'} \ar@{->}[r]^{\theta_{Q}} &     Q_+\cO_{\bbY'}
                }$$
                where the upper horizontal arrow equals the map $\sD^{\dagger}_{\cP}\ot\theta_{v}^{alg}$.
\end{cor}
\begin{proof} As $G$ is c-proper, $G_+=G_!$, and using~\ref{compatib}, we can see 
that $\theta_Q=G_+\circ\theta_{J'}$. Similarly, we have the equality $v_{\bbQ}=g_{\bbQ}\circ j'_{\bbQ}$, and as 
$g_{\bbQ}$ is proper, $\theta_{v_{\bbQ}}=g_{\bbQ +}\circ \theta_{j'_{\bbQ}}.$ We finally use the compatibility
 for projective morphisms ~\ref{compat_proj}, and we observe that we obtain the diagram of the corollary after 
applying $\hg_+$ to the previous diagram ~\ref{thetacomp0}.
\end{proof}

Remark: We have the identifications $\theta_{j'_{\bbQ}}({\cO_{Y_{\bbQ}}})\simeq \cO_{X'_{\bbQ}}$ and 
$\theta_{J'}\cO_{\frX',\Q}({^\dagger Z'_s})\simeq \cO_{\frX',\bbQ}$, and that 
$$  \sD^{\dagger}_{\frX'}\ot_{\osD_{X'_{\bbQ}}}\overline{\cO_{X'_{\bbQ}}}\simeq \cO_{\frX',\bbQ}.$$

\vskip5pt

We have the constant overholonomic modules on $\bbY$ resp. $\bbY'$
$$\cO_{\bbY}=R\Gamma_{\bbY}(\cO_{\cP,\Q})[d] \hskip20pt \text{resp.}\hskip20pt \cO_{\bbY'}=R\Gamma_{\bbY'}(\cO_{\frX',\Q})=\cO_{\frX',\Q}({^\dagger Z'_s})$$
as defined in \ref{def_constant}, where $d=\dim P_s-\dim Y_s$.

\begin{lemma} \label{uplus}There are canoncical isomorphisms \begin{enumerate}\item $u^! \cO_{\bbY}\simeq\cO_{\bbY'}$,
            \item $u_+\cO_{\bbY'}\simeq\cO_{\bbY}$.
             \end{enumerate}
\end{lemma}
\begin{proof} 
By ~\cite[1.2.8]{AbeCaro}, $u^!$ and $u_+$ are exact functors of the categories Ovhol($\bbY$) and
Ovhol($\bbY'$), and quasi-inverse, so that (ii) is a direct consequence of (i). Recall also that, by 
~\cite[2.2.6.1,2.2.8,2.2.14]{CaroSurcohFL}, $R\Gamma_{\bbY'}\circ R\Gamma_{\bbY'}=R\Gamma_{\bbY'}$. 
We compute 
\begin{align*}
u^!\cO_{\bbY} &=R\Gamma_{\bbY'} \circ \hg^! (R\Gamma_{X_s}^{\dagger}(^\dagger Z_s)(\cO_{\cP,\Q})[d])\\
               &\simeq R\Gamma_{\bbY'}\circ R\Gamma_{X'_s}^{\dagger}\circ(^\dagger Z'_s)\hg^!\cO_{\cP,\Q}[d] \textrm{ \;
\cite[Théorème 2.2.18]{CaroSurcohFL}}\\
             &\simeq R\Gamma_{\bbY'}\circ R\Gamma_{\bbY'} \cO_{\frX',\Q} \\
               &\simeq R\Gamma_{\bbY'} \cO_{\frX',\Q}.
\end{align*}
\end{proof}

We come to the main result, which describes the relation between the classical intermediate extension 
$v_{\bbQ!+}\cO_{Y_{\bbQ}}$ on the generic fibre and the Abe-Caro intermediate extension $v_{!+}\cO_{\bbY}$ on the special fibre. 

\begin{thm}\label{mainappli} There is a canonical isomorphism 
$$  \sD^{\dagger}_{\cP}\ot_{\osD_{P_{\bbQ}}}\overline{v_{\bbQ!+}(\cO_{Y_{\bbQ}})}\simeq v_{!+}(\cO_{\bbY}).$$
\end{thm}
\begin{proof}  Again, by \cite[1.2.8]{AbeCaro}, $u_+=u_!$, and $\theta_Q=\theta_v\circ u_+$. By previous
lemma~\ref{uplus}, $u_+\cO_{\bbY'}\simeq\cO_{\bbY}$ and we have a commutative diagram 
$$ \xymatrix{ v_+\cO_{\bbY} \ar@{->}[d]^{\simeq}\ar@{->}[r]^{\theta_{v}}& v_!\cO_{\bbY} \ar@{->}[d]^{\simeq}\\
     Q_+\cO_{\bbY'}\ar@{->}[r]^{\theta_{Q}} & Q_!\cO_{\bbY'}
.}$$
Now we have 
\begin{align*} v_{!+}(\cO_{\bbY}) &= \im(\theta_{v}) \\ 
                                  &\simeq \im(\theta_{Q}) \\
                      &\simeq \sD^{\dagger}_{\cP}\ot_{\osD_{P_{\bbQ}}}\im(\theta_{v}^{alg})
\textrm{ by ~\ref{thetacomp}},\\
                              & \simeq \sD^{\dagger}_{\cP}\ot_{\osD_{P_{\bbQ}}}\overline{v_{\bbQ!+}(\cO_{Y_{\bbQ}})}.
\end{align*}
\end{proof}

\section{Localization theory on the flag variety}\label{sec_localizationflagvariety}

We specialize the above theory to the case where $\cP$ is the (formal) flag variety of a connected split reductive group $G$ over $\fro$. 
Such a space is coherently $\sD^{\dagger}$-affine and its algebra of global differential operators $H^0(\cP, \sD_{\cP}^{\dagger})$ identifies with (a central reduction of) the crystalline distribution algebra of $G$. Truly in the spirit of classical localization theory \cite{BB81}, this allows us to analyze geometrically the module theory of the distribution algebra.

\subsection{Crystalline distribution algebras}
Let $G$ be a connected split reductive group scheme over $ \fro$.
Let $I$ be the kernel of the morphism $\fro$-algebras $\varep_G: \fro[G]\rig \fro$ which represents $1\in G$.
Then $I/I^2$ is a free $\fro=\fro[G]/I$-module of finite rank. Let $t_1,\ldots,t_N\in I$ whose classes
modulo $I^2$ form a base of $I/I^2$. The $m$-PD-envelope of $I$ is denoted by $P_{(m)}(G)$. This algebra is a free $\fro$-module
with basis $$\ut^{\{\uk\}}=t_1^{\{k_1\}}\cdots t_N^{\{k_N\}},$$
where $q_i!t_i^{\{k_i\}}=t_i^{k_i}$ with $i=p^mq_i+r$ et $r<p^m$ \cite[1.5]{BerthelotDI}.
The algebra $P_{(m)}(G)$ has a descending filtration by the ideals

\begin{gather} \label{PDfilt}
 I^{\{n\}}=\bigoplus_{|\uk|\geq n} \fro \cdot \ut^{\{\uk\}}.
\end{gather}
The quotients $P^n_{(m)}(G):= P_{(m)}(G)/I^{\{n+1\}}$ are generated, as $\fro$-module, by the elements $\ut^{\{k\}}$ where
$|\uk|\leq n$ and there is an isomorphism $P^n_{(m)}(G)\simeq \bigoplus_{|\uk|\leq n} \fro \ut^{\{\uk\}}$ as $\fro$-modules.
There are canonical surjections $pr^{n+1,n}: P^{n+1}_{(m)}(G)\trig P^n_{(m)}(G)$.

\vskip5pt 

We note $$\Lie(G):=\Hom_\fro(I/I^2,\fro).$$ The Lie-algebra $\Lie(G)$ is a free $\fro$-module with
basis $\xi_1,\ldots,\xi_N$ dual to $t_1,\ldots,t_N.$ For $m'\geq m$, the universal property of divided power algebras gives
homomorphismes of filtered algebras
$\psi_{m,m'}\,\colon \, P_{(m')}(G)\rig P_{(m)}(G)$
which induce on quotients homomorphismes of algebras
$\psi_{m,m'}^n \,\colon \,P^n_{(m')}(G)\rig P^n_{(m)}(G).$
 {\it The module of distributions of level $m$ and order $n$} is
$D_n^{(m)}(G):=\Hom_{\fro}(P^n_{(m)}(G),\fro)$
{\it The algebra of distributions of level $m$} is defined to be
$$D^{(m)}(G):=\varinjlim_n D_n^{(m)}(G)$$ where the limit is taken with respect to the maps $\Hom_{\fro}(pr^{n+1,n},\fro)$.

\vskip5pt

For $m'\geq m$, the algebra homomorphisms $\psi_{m,m'}^n$ give dually linear maps $\Phi_{m,m'}^n$ : $D^{(m)}_n(G)\rig D^{(m')}_n(G)$ and finally a morphism of filtered algebras
$\Phi_{m,m'}: D^{(m)}(G)\rig D^{(m')}(G).$ The direct limit 

$${\rm Dist}(G)=\varinjlim_m D^{(m)}(G)$$

equals the classical distribution algebra of the group scheme $G$ \cite[II.\S 4.6.1]{DemazureGabriel}.

\vskip5pt

Let now $\GG$ be the completion of $G$ along its special fibre.
 We write $G_i=\Spec\; \fro [G]/\pi^{i+1}$. The morphism $G_{i+1}\hrig G_i$ induces
$D^{(m)}(G_{i+1})\rightarrow D^{(m)}(G_i)$. We put
$$\widehat{D}^{(m)}(\GG):=\varprojlim_{i}D^{(m)}(G_i).$$ If $m'\geq m$, one has the morphisms $\hat{\Phi}_{m,m'}: \widehat{D}^{(m)}(\GG)\rig
\widehat{D}^{(m')}(\GG)$ and the {\it crystalline distribution algebra} is defined to be 
$$
D^{\dagger}(\GG):= \varinjlim_{m}\widehat{D}^{(m)}(\GG) \otimes\Q.$$

Note, as for differential operators, that this dagger-algebra appears with coefficients tensored by $\Q$. 
For more details on the basic theory of the algebra $D^{\dagger}(\GG)$ we refer to \cite{HS1,HS2}. 

\vskip5pt

For a character $\theta: Z(\frg)\rightarrow K$ of the center $Z(\frg)$ of the universal enveloping algebra of the $K$-Lie algebra
$\frg=\Lie(G)\otimes\Q,$ we will always denote by
 $$D^{\dagger}(\GG)_{\theta}:= D^{\dagger}(\GG)_{\theta}\otimes_{Z(\frg),\theta}K $$ the corresponding central reduction of $D^{\dagger}(\GG)$. 
 The {\it trivial} character is the character $\theta_0$ with $\ker \theta_0 = Z(\frg) \cap (U(\frg)\frg)$.

\vskip5pt

\subsection{The localization theorem and holonomicity}

 Let in the following $\theta=\theta_0$ be the trivial character. 
Our goal is to analyze the {\it central block} of the category of $D^{\dagger}(\GG)$-modules, i.e. the category of
$D^{\dagger}(\GG)_{\theta_0}$-modules. We keep the notation from the preceding section. 

\vskip5pt 

We let $B\subset G$ be a Borel subgroup containing a maximal split torus $T$, with unipotent radical $N$. Denote by
$$P:=G/B$$ the flag scheme. It is a smooth and projective scheme over $\fro$. We denote by $\cP$ its formal completion. The $G$-action on $P$ by translations endowes $\cP$ with a $\cG$-action. We recall the localization theorem for arithmetic $\sD$-modules on the flag variety. 

\begin{thm} \label{localizationthm}
{\rm (a)} The global section functor induces an equivalence of categories between coherent $\sD^{\dagger}_{\cP}$-modules and coherent $H^0(\cP,\sD^{\dagger}_{\cP})$-modules. A quasi-inverse is given by the functor 
$${\mathscr Loc}(M)= \sD^{\dagger}_{\cP}\otimes_{H^0(\cP,\sD^{\dagger}_{\cP})}M.$$

\vskip5pt

{\rm (b)} The $\cG$-action on $\cP$ induces an algebra isomorphism
$$D^{\dagger}(\GG)_{\theta_0}\car H^0(\cP, \sD^{\dagger}_{\cP}).$$

\end{thm}
\begin{proof} This summarizes the main results of \cite{HS2} and \cite{NootHuyghe09}. 
\end{proof}

Remark: Sarrazola-Alzate has extended the above theorem to the case of an arbitrary central character $\theta$ using a twisted version of the sheaf 
$\sD^{\dagger}_{\cP}$, as in the classical setting, cf. \cite{SA1}.

\vskip5pt 

We specialize the classification result \ref{classify} to the case of the flag variety $\cP$.
First of all, $\cP$ is quasi-projective (in fact projective) over $\fro$. 
According to the main result of \cite{CaroStab} the notions {\it overholonomic} and {\it holonomic}
coincide for $F$-$\sD^{\dagger}_{\cP}$-modules. Hence we have the equality  

$$F\text{-Ovhol}(\bbP/K)= \{ \text{ holonomic $F$-}\sD^{\dagger}_{\cP}\text{-modules }\}$$

inside the category of coherent $\sD^{\dagger}_{\cP}$-modules. This motivates the following definition. 

\begin{dfn} A $D^{\dagger}(\GG)_{\theta_0}$-module $M$ is called 
{\it geometrically $F$-holonomic} if ${\mathscr Loc}(M)\in F\text{-Ovhol}(\bbP/K)$.
\end{dfn}

\vskip5pt

Recall from \ref{subsec_ovholext} the set of equivalence classes of pairs $(Y,\cE)$ where $Y\subset \cP_s$ 
is a connected smooth locally closed subvariety and $\cE$ is an irreducible overconvergent $F$-isocrystal on $\bbY=(Y,X)$. We put $\cL(Y,\cE):= v_{!+}(\cE)\in F\text{-Ovhol}(\bbP)$ where $v: \bbY\rightarrow\bbP$ is the immersion of couples associated with $Y$.

\begin{thm}\label{classRep}
 The correspondence $(Y,\cE)\mapsto H^0(\cP,\cL(Y,\cE))$ induces a bijection 
$$\{\text{equivalence classes of pairs $(Y,\cE)$}\}\car \{\text{irreducible $F$-holonomic $D^{\dagger}(\GG)_{\theta_0}$-modules} \}/ {\simeq} $$
 
\end{thm}
\begin{proof} 
This follows from the classification theorem \ref{classification} together with \ref{localizationthm}.
\end{proof}

We point out a related interesting property of the category of holonomic $F$-$\sD^{\dagger}_{\cP}$-modules.

\vskip5pt

It is conjectured by de Jong that, if $X$ is a connected smooth projective variety over an algebraically closed field of characteristic $p>0$ 
with trivial \'etale fundamental group, then any isocrystal on $X$ is constant. 
This conjecture is proved under certain additional assumptions by Esnault-Shiho in \cite{EsSh}. 
In our case, the fibration $G\rightarrow G/B=P$ is a separable proper morphism with geometrically connected fibre between locally noetherian connected schemes.
To compute the fundamental group of $\cP_s$, we may pass to a simply connected cover of the semisimple derived group of $\cG_s$. The homotopy exact sequence \cite[Exp. 10 Cor. 1.4]{SGA1} implies then that \'etale fundamental group of $\cP_s$ is trivial. Here is a short representation-theoretic proof of de Jong's conjecture for the flag variety $\cP_s$.\footnote{The homotopy exact sequence implies in the same manner 
that the generic fibre $P_K$ has trivial \'etale fundamental group. By Chern-Weil theory and Grothendieck's theorem on formal functions, the de Rham Chern classes on $P_K$ become trivial after tensoring with $\bbQ$. But these classes correspond to the rational crystalline classes on $\cP_s$ via the comparison theorem between de Rham and crystalline cohomology, from which one may deduce the conjecture. 
We thank H. Esnault for explaining this general argument to us. }

\begin{prop} \label{deJong} Any convergent isocrystal on $\cP_s$ is constant. 
\end{prop}
\begin{proof}
Any convergent isocrystal $\cE$ may be viewed as a coherent $\sD^{\dagger}_{\cP}$-module which is coherent over $\cO_{\cP,\bbQ}$
\cite[Prop. (4.1.4)]{BerthelotDI}. Then $H^0(\cP,\cE)$ is a finite dimensional representation of the reductive $K$-Lie algebra $\frg$ and 
hence completely reducible (semisimple). In addition, it has central character $\theta_0$. But the trivial one dimensional representation is the only irreducible $\frg$-representation 
of finite dimension and with central character $\theta_0$. Since the trivial representation localizes to the trivial connection $\cO_{\cP,\bbQ}$ and since localization commutes 
with direct sums, the isocrystal $\cE$ must be constant.
\end{proof}

\section{Highest weight representations and the rank one case}\label{secfinal}

We keep the notation from the preceding section. We assume from now on that the field $K$ is locally compact. 

\subsection{Highest weight representations}

We establish a crystalline version of the central block of the classical BGG category $\cO$ and show
that its objects are geometrically $F$-holonomic. We then compute their associated parameters $(Y,\cE)$ in the geometric classification \ref{classRep}. 

\vskip5pt

Let $\Delta$ be the set of
simple roots in $\Phi^{+}$. We fix a (Chevalley) basis for $\Lie(G)$ compatible with its root space decomposition. In particular, we obtain a 
$\fro$-basis $t_1,...,t_n$ of $\Lie(T)$ which is made up from a $K$-basis of the center of $\frg$ and 
finitely many elements $t_{\alpha}$, indexed by $\alpha\in \Delta$, such that $\beta(t_{\alpha})\in \bbZ$ for all $\beta\in\Phi$. 
Let $\Gamma:= \bbZ_{\geq
0}\Phi^+\subset \bbQ \Phi=:\Lambda_r\subseteq\Lambda$
where $\Lambda_r$ and $\Lambda$ are the root lattice and the integral weight lattice respectively. 

\vskip5pt

For $w\in W$ we let $\lambda_w = -w(\rho) -\rho$. These are $|W|$ pairwise different elements of $\Lambda_r$.

\vskip5pt

Let $\cO_0$ be the central block of the classical BGG category, e.g. \cite{H1}. This is a full abelian subcategory 
of finitely generated $U(\frg)_{\theta_0}$-modules which is noetherian and artinian. Its irreducible objects are given by the unique irreducible quotients $M(\lambda_w)\rightarrow L(\lambda_w)$ where $$M(\lambda_w):=U(\frg)\otimes_{U({\frt}),\lambda_w} K$$
is the Verma module with highest weight $\lambda_w$ for $w\in W$.

\vskip5pt

To define a crystalline variant of the category $\cO_0$ we follow the constructions given in \cite{SchmidtBGG} in the case of the Arens-Michael envelope of $U(\frg)$. In order to do so, we need the field $K$ to be locally compact.  

\vskip5pt

By the discussion in \cite[5.3]{HS1} the algebra $D^{\dagger}(\GG)=\varinjlim_{m}\widehat{D}^{(m)}(\GG) \otimes\Q$ is an inductive limit of Hausdorff locally convex $K$-vector spaces with injective and compact transition maps. According to \cite[7.19/16.9/16.10]{NFA} it is therefore Hausdorff, complete and barrelled. 

\vskip5pt 

The framework of diagonalisable modules over suitable commutative topological $K$-algebras as described in \cite[sec. 2]{SchmidtBGG} applies therefore to the $K$-algebra $D^{\dagger}(\TT)$. Note that it contains the universal enveloping algebra 
$U(\frt)$ as a dense subalgebra. A $K$-{\it valued weight} $\lambda$ of $D^{\dagger}(\TT)$ is a 
$K$-algebra homomorphism $D^{\dagger}(\TT)\rightarrow K$. A set of weights
$Y$ is called {\it relatively compact} if its image under the injective map
$\lambda\mapsto
(\lambda(t_1),...,\lambda(t_n))$ has a compact closure 
in $K^n$. Let $\lambda$ be weight and $M$ some topological $D^{\dagger}(\TT)$-module.
A nonzero $m\in M$
is called a {\it$\lambda$-weight vector} if $h.m=\lambda(h).m$ for
all $h\in D^{\dagger}(\TT)$. In this case $\lambda$ is called a {\it weight of
$M$}. The closure $M_\lambda$ in $M$ of the $K$-vector space
generated by all $\lambda$-weight vectors is called the
$\lambda$-{\it weight space} of $M$. The module $M$ is called
$D^{\dagger}(\TT)$-{\it diagonalisable} if there is a set of weights
$\Pi(M)$ with the property: to every $m\in M$ there
exists a family $\{m_\lambda\in M_\lambda\}_{\lambda\in\Pi(M)}$
converging cofinitely against zero in $M$ and satisfying
$$m=\sum_{\lambda\in\Pi(M)}m_\lambda.$$
Given a diagonalisable module $M$ we may form $M^{ss}=\oplus_{\lambda\in\Pi(M)}M_\lambda$
(depending on the choice of $\Pi(M)$).

\vskip5pt 

\begin{dfn}\label{def} The category $\cO_0^{\dagger}$ equals the
full subcategory of $D^{\dagger}(\GG)_{\theta_0}$-modules $M$ satisfying:
\begin{itemize}
 \item[(1)] $M$ is a coherent $D^{\dagger}(\GG)_{\theta_0}$-module
    \item[(2)] $M$ is $D^{\dagger}(\TT)$-diagonalisable with $\Pi(M)$ contained in the union of the cosets $\lambda_w-\Gamma $
    \item[(3)] All weight spaces $M_\lambda, \lambda\in\Pi(M)$, are finite
    dimensional over $K$.
\end{itemize}
\end{dfn}

By definition, given $M\in\cO_0^{\dagger}$, then
any finitely generated $U(\frg)$-submodule of $M^{ss}$ lies in $\cO_0$. In particular, $M^{ss}$ contains a {\it maximal vector}, i.e. a nonzero 
$m\in M_{\lambda}$ (of some weight $\lambda$) such that $\frn.m=0$. We will make precise the relation between 
the two categories $\cO_0$ and $\cO_0^{\dagger}$ below. 

\vskip5pt 

We list some basic properties of the category $\cO_0^{\dagger}$.

\begin{prop}\label{propO}
\begin{itemize}
    \item[(i)] The direct sum of two modules of $\cO_0^{\dagger}$ is in $\cO_0^{\dagger}$
    \item [(ii)] the (co)kernel and (co)image of an arbitrary $D^{\dagger}(\GG)_{\theta_0}$-linear
    map between objects in $\cO_0^{\dagger}$ is in $\cO_0^{\dagger}$
    \item[(iii)] the sum of two coherent submodules of an object in $\cO_0^{\dagger}$ is in $\cO_0^{\dagger}$
    \item[(iv)] any finitely generated submodule of an object in
    $\cO_0^{\dagger}$ is in $\cO_0^{\dagger}$
    \item[(v)] $\cO_0^{\dagger}$ is an abelian category.

     \end{itemize}
\end{prop}
\begin{proof} This can be proved using a variant of the proof of \cite[Prop. 3.6.3]{SchmidtBGG}. Note that any $\Pi(M)$ which is contained in the union of the cosets $\lambda_w-\Gamma$ is relatively compact. Indeed, $\Gamma$ is relatively compact its closure 
being contained in the compact subset $\bbZ_p^{|\Delta|}$ of $K^n$, cf. \cite[Lem. 3.6.1]{SchmidtBGG}. 
\end{proof}


We exhibit Verma type modules in $\cO_0^{\dagger}$. The main difference between the case of the crystalline distribution algebra and the case 
of the Arens-Michael envelope treated in \cite{SchmidtBGG} is that {\it not} every weight $\frt\rightarrow K$ extends to a weight of $D^{\dagger}(\TT)$. The following lemma is sufficient for our purposes. 

\begin{lemma} 
Any linear form $\lambda: \Lie(T)\rightarrow \fro$ such that $\lambda(h_i)\in\bbZ_p$ for all $i=1,...,n$ extends canonically to a 
$K$-algebra homomorphism $D^{\dagger}(\TT) \rightarrow K$.
\end{lemma}
\begin{proof}
Recall that the distribution algebra
${\rm Dist}(\bbG_m)$ of the $\fro$-group scheme $\bbG_m$ is generated as an $\fro$-module by the elements $\binom{\delta_1}{k}$ for $k\in\bbN$
where $\delta_1$ is a generator of $\Lie(\bbG_m)$, cf. \cite[Part I.7.8]{Jantzen}. Our choice of Chevalley basis implies an isomorphism of group schemes
$T\simeq \prod_{i=1,...,n} \bbG_m$ such that the basis element $h_i$ becomes the generator of the $i$-th copy $\Lie(\bbG_m)$.
Since $\binom{\lambda(h_i)}{k}\in\bbZ_p$, the associated $K$-algebra homomorphism $\lambda: U(\frt)\rightarrow K$ restricts to an $\fro$-algebra homomorphism ${\rm Dist}(T)\rightarrow\fro$. Since 
${\rm Dist}(T)=\varinjlim_m D^{(m)}(T)$, this extends then to a $K$-algebra homomorphism $D^{\dagger}(\TT) \rightarrow K$, 
\end{proof}

We may apply the lemma to any weight $\lambda_w$ and hence consider the $D^{\dagger}(\GG)$-module 

$$M^{\dagger}(\lambda_w):=D^{\dagger}(\GG) \otimes_{D^{\dagger}(\TT),\lambda_w} K.$$

\begin{prop}\label{verma} The module $M^{\dagger}(\lambda_w)$ lies in $\cO_0^{\dagger}$. We have 
$$M^{\dagger}(\lambda_w)^{ss}=M(\lambda_w)\hskip10pt \text{and} \hskip10pt M^{\dagger}(\lambda_w)=D^{\dagger}(\GG) \otimes_{U(\frg)} M(\lambda_w).$$

 There is a canonical inclusion
preserving bijection between subobjects of $M^{\dagger}(\lambda_w)$ and abstract
$U(\frg)$-submodules of $M(\lambda_w)$. In particular, $M^{\dagger}(\lambda_w)$ admits a unique maximal subobject and hence a unique irreducible quotient 
$L^{\dagger}(\lambda_w)$. The latter satisfies $L^{\dagger}(\lambda_w)^{ss}=L(\lambda_w)$. 

\end{prop}
\begin{proof}
This can be proved as in \cite[Prop. 3.7.1]{SchmidtBGG}. Note that the triangular decomposition
$D^{(m)}(G)=D^{(m)}(N^{-})\otimes_\fro D^{(m)}(T)\otimes_\fro D^{(m)}(N)$, cf. \cite[2.2]{HS2},
implies that $M^{\dagger}(w)\simeq D^{\dagger}(\NN^{-})$ as a left $D^{\dagger}(\NN^{-})$-module. 
This implies the first displayed identity. 
Moreover, $M^{\dagger}(\lambda_w)$ equals the quotient of $D^{\dagger}(\GG)$ by the left ideal generated by $\ker(\lambda_w)$,
which implies the second displayed identity. Note also that the nonzero quotient morphism $M^{\dagger}(\lambda_w)\rightarrow L^{\dagger}(\lambda_w)$
yields a nonzero quotient morphism $M^{\dagger}(\lambda_w)^{ss}\rightarrow L^{\dagger}(\lambda_w)^{ss}$ since $(-)^{ss}$ is faithful and exact 
\cite[Prop. 2.0.2]{SchmidtBGG}. Hence $M^{\dagger}(\lambda_w)^{ss}=M(\lambda_w)$ implies $L^{\dagger}(\lambda_w)^{ss}=L(\lambda_w)$.
\end{proof}

\begin{cor} The modules $L^{\dagger}(\lambda_w)$ exhaust, up to isomorphism, all the irreducible objects in $\cO_0^{\dagger}$. 
\end{cor}
\begin{proof} Let $L$ be an irreducible object in $\cO_0^{\dagger}$. Take a maximal vector $m\in L^{ss}$ of some weight $\lambda$.
 Then $U(\frg)m$ is a highest weight module in $\cO$ of weight $\lambda$, cf. \cite[1.2]{H1}. Hence $Z(\frg)$ acts on the maximal vector $m$ via the central character $\theta_{\lambda}$ associated to $\lambda$ via the Harish-Chandra homomorphism \cite[1.7]{H1}. But $U(\frg)m\subset L$ whence $\theta_{\lambda}=\theta_0$ and so $\lambda=\lambda_w$ for some $w\in W$.
 We obtain a nonzero $D^{\dagger}(\GG)$-linear map $M^{\dagger}(\lambda_w)\rightarrow L, 1\otimes 1\mapsto m$. 
 So $L$ is an irreducible quotient of $M^{\dagger}(\lambda_w)$, i.e. $L\simeq L^{\dagger}(\lambda_w)$.
\end{proof}

\begin{cor} The category $\cO_0^{\dagger}$ is artinian and noetherian. 
\end{cor}
\begin{proof}
This can be deduced similarly to \cite[Prop.4.2.2]{SchmidtBGG}.
In fact, let $M\in\cO_0^{\dagger}$ and consider the finite-dimensional $K$-vector space
$V:=\sum_{w} M_{\lambda_w}$. Suppose $N'\varsubsetneq
N\subseteq M$ are two subobjects. Let $m\in N\setminus N'$ be a maximal
vector of some weight $\lambda$. As in the preceding proof we deduce from the action of $Z(\frg)$ on $m$ 
that $\lambda=\lambda_w$ for some $w\in W$. So $m\in N\cap V$ whence $\dim_K N\cap V>\dim_K N'\cap V$.
This implies that $M$ has finite length.
\end{proof} 

Given a module $M\in\cO_0$, we can define the coherent $D^{\dagger}(\GG)_{\theta_0}$-module 

$$M^{\dagger} := D^{\dagger}(\GG) \otimes_{U(\frg)} M.$$

 \begin{thm}\label{ffembedding}
The functor $F: M\rightsquigarrow M^\dagger$ is exact and induces an equivalence of abelian categories
$$ \cO_0 \car \cO_0^{\dagger}.$$ A quasi-inverse is given by the functor $(-)^{ss}$.
\end{thm}
\begin{proof} The ring extension $U(\frg)\rightarrow D^{\dagger}(\GG)$ is flat \cite[Lem. 4.1]{HS2}.
We already now that $F(M(\lambda_w))=M^{\dagger}(\lambda_w)$.
Since any object $M\in\cO_0$ admits a finite composition series 
with irreducible constituents of the form $L(w)$, there is a surjection $\oplus_w M(\lambda_w) \rightarrow M$. 
Since $F$ commutes with direct sums, we see that $F(M)$ equals the quotient of  $\oplus_w M^{\dagger}(\lambda_w)$ modulo a finitely generated submodule and so lies in $\cO_0^{\dagger}$, according to parts (iii)-(v) of \ref{propO}. We therefore have an exact functor $F:\cO_0 \rightarrow \cO_0^{\dagger}.$
Given $M\in\cO_0^{\dagger}$ we have a functorial morphism $M\rightarrow F(M)^{ss}, m\mapsto 1\otimes m$ which is bijective for irreducible $M$ according to \ref{verma}. By d\'evissage, we obtain $M\simeq F(M)^{ss}$ in general. Let $M\in \cO_0^{\dagger}$. To obtain $M^{ss}\in\cO_0$ we use induction on the length of $M$ and suppose that $N\subset M$ is a maximal submodule, i.e. $M/N \simeq L^{\dagger}(\lambda_w)$ for some $w$, such that $N^{ss}\in\cO_0$.
Exactness of $(-)^{ss}$ and $L^{\dagger}(\lambda_w)^{ss}=L(\lambda_w)$ implies that $M^{ss}$ is an extension of two finitely generated $U(\frg)$-modules and hence itself 
finitely generated.  So $M^{ss}\in\cO_0$. We may now deduce that $(-)^{ss}$ is also a 
right quasi-inverse to $F$. Indeed, for any $M\in\cO_0^{\dagger} $, there is a natural morphism $F(M^{ss})\rightarrow M$ in $\cO_0^{\dagger}$ which is bijective for irreducible $M$ according to \ref{verma}. By d\'evissage, we obtain $F(M^{ss})\simeq M$ in general.
\end{proof}

To finish this section, we will show that the irreducible modules $L^{\dagger}(\lambda_w)$ are all geometrically $F$-holonomic. 

\vskip5pt 

To do this, fix $w\in W$ and let $$Y_w:=BwB/B \subset P=G/B$$ be the Bruhat cell in $P$ associated with $w\in W$ . Let  $v: Y_w\hookrightarrow P$ be the corresponding immersion over $\fro$ and let 
 $v_{\Q}: Y_{w\Q}\hookrightarrow P_{\Q}$ be the corresponding immersion on the level of $K$-algebraic varieties. It is well-known 
(e.g. \cite[Prop. 12.3.2]{Hotta}) that there is a canonical isomorphism of 
 $\sD_{P_{\bbQ}}$-modules  $${\rm Loc}(L(\lambda_w)):=\sD_{P_{\bbQ}}\otimes_{U(\frg)} L(\lambda_w) \simeq v_{\bbQ!+}(\cO_{Y_{w\bbQ}}).$$

Let  $X_w \subset P$ be the Zariski closure of the Bruhat cell $Y_w$ in $P$, a Schubert scheme. We let $$X'_w\longrightarrow X_w$$ be its {\it Demazure desingularization}, which is defined at the level of $\fro$-schemes ~\cite[II, 13.6]{Jantzen}. We are then in the axiomatic situation (S), the point of departure for subsection \ref{sec_compatibillity}, so that all the results of this subsection apply. In particular, we have the frame  $\bbY_w=(Y_{w,s},X_{w,s},\cP)$ together with its $c$-locally closed immersion $$v: \bbY_w\longrightarrow \bbP$$ and the constant overholonomic module $\cO_{\bbY_w}$ on $\bbY_w$.
Its intermediate extension $v_{!+}(\cO_{\bbY_w})$ is an irreducible holonomic  $F$-$\sD^{\dagger}_{\cP}$-module.
In this situation, the main theorem \ref{mainappli} implies directly the following result. 

\begin{prop}\label{app_flag} There is a canonical isomorphism of $\sD^{\dagger}_{\cP}$-modules   $$ \sD^{\dagger}_{\cP}\ot_{\osD_{P_{\bbQ}}}
\overline{v_{\bbQ!+}(\cO_{Y_{w\bbQ}})}\simeq v_{!+}(\cO_{\bbY_w}) .$$
\end{prop}

Now consider the localization

$$ {\mathscr Loc} (L^{\dagger}(\lambda_w)) = \sD^{\dagger}_{\cP}\otimes_{D^{\dagger}(\GG)}  L(\lambda_w)^{\dagger}.$$
\begin{thm}\label{thm_final} One has a canonical isomorphism of $\sD^{\dagger}_{\cP}$-modules
$$ {\mathscr Loc} (L^{\dagger}(\lambda_w)) \simeq v_{!+}(\cO_{\bbY_w}).$$
The modules $L^{\dagger}(\lambda_w)$ are geometrically $F$-holonomic for all $w\in W$. 
\end{thm}
\begin{proof} We write $L(w)$ resp. $L^{\dagger}(w)$ for $L(\lambda_w)$ resp. $L^{\dagger}(\lambda_w)$.
Since $L^{\dagger}(w)= D^{\dagger}(\GG) \otimes_{U(\frg)} L(w)$, associativity of tensor products yields a canonical isomorphism 

$$\sD^{\dagger}_{\cP}\otimes_{D^{\dagger}(\GG)}  L(w)^{\dagger}\simeq \sD^{\dagger}_{\cP}\otimes_{U(\frg)}  L(w)
\simeq \sD^{\dagger}_{\cP} \otimes_{\osD_{P_\bbQ}} (\osD_{P_\bbQ}\otimes_{U(\frg)}  L(w)) \simeq \sD^{\dagger}_{\cP} \otimes_{\osD_{P_\bbQ}} \overline{{\rm Loc}(L(w))}.$$
Since ${\rm Loc}(L(w))\simeq v_{\bbQ!+}(\cO_{Y_{w\bbQ}})$, the asserted isomorphism follows now in combination with \ref{app_flag}.
Since $v_{!+}(\cO_{\bbY_w})$ is a holonomic  $F$-$\sD^{\dagger}_{\cP}$-module, the module $L^{\dagger}(w)$ is seen to be geometrically $F$-holonomic.
\end{proof}

\subsection{The ${\rm SL}_2$-case}

We suppose $G={\rm SL}_2$. We let $B$ be the subgroup of upper triangular matrices and $T\subset B$ be the subgroup of diagonal matrices. We identify 
$\Lambda=\bbZ$ so that $\Delta=\{\alpha\}$ with $\alpha=2$. We identify $$P=G/B=\bbP_{\fro}^1$$ with the projective line $\bbP_{\fro}^1$ over $\fro$. We choose an affine coordinate $t$ around zero. The group $G$ acts by fractional transformations 

$$\left(\begin{array}{rr}
a & b  \\
c & d \\
\end{array}\right).\; (t)= \left(\frac{at+b}{ct+d}\right)$$

\vskip5pt 
in the usual way. The stabiliser of the point $\infty\in\bbP_{\fro}^1$ is $B$.

\vskip5pt

We note that {\it any} irreducible $\sD^{\dagger}_{\cP}$-module is holonomic, since $\cP$ has dimension one. 
 In this setting, the theorem \ref{classRep}
amounts to a classification of {\it all} irreducible $F-D^{\dagger}(\GG)_{\theta_0}$-modules in terms of irreducible 
overconvergent $F$-isocrystals $\cE$ on couples $\bbY=(Y,X)$ where $Y$ is either :

\vskip5pt 

(1) a closed point of $\bbP_{k}^1$ or\\

 (2) an open complement of finitely many closed points $Z=\{y_1,...,y_n\}$ of $\bbP_{k}^1$.\\

 In case (1), the point is a complete invariant, since we have necessarily $\cE=\cO_{\bbY}$ in this case. Suppose that the point is $k$-rational.
 Since the (finitely many) $k$-rational points $\bbP^1(k)$ of $\bbP_{k}^1$ form a single orbit under the natural action of the (finite) group $G(k)$ of $k$-rational points of $G$, it suffices to consider the point 
 $$\{\infty\}=Y_{1,s}=X_{1,s}$$ in $\bbP_{k}^1$. According to \ref{thm_final}, the global sections of $v_{!+}(\cO_{\bbY})$ are 
 equal to the $D^{\dagger}(\GG)_{\theta_0}$-module $L^{\dagger}(-2)$, the crystalline version of the classical anti-dominant Verma module $M(-2)=L(-2)$.
 
Suppose now that the the point is $k'$-rational for a finite extension field $k'/k$. Let $M=H^0(\cP, v_{!+}(\cO_{\bbY}))$.
Let $\fro'$ be a finite extension of $\fro$ with residue field $k'$ and quotient field $K'$. The base change $M_{K'}=M \otimes_K K'$ has the same geometric parameter, but now considered a rational point of the special fibre of $\cP \times_{\fro} \fro'$. This means that $M$ is a twisted form of the module 
$L^{\dagger}(-2)$, with respect to the field extension $K'/K$. 

\vskip5pt

We come to case (2). For $Z=\varnothing$ and hence $Y=\bbP_{k}^1$ we obtain the trivial representation, i.e. the augmentation character $D^{\dagger}(\GG)\rightarrow K$. Indeed, there are no convergent $F$-isocrystals on $\cP$ besides the constant one, cf. prop. \ref{deJong}. Let $n>0$. Modulo the appearance of twisted forms (see the above argument), we may assume that all points $y_1,...,y_n$ are $k$-rational and $y_1=\infty$. 
There are then two extreme cases $$ Y=\bbA^1_k \hskip10pt \text{resp.} \hskip10pt Y= \bbP_{k}^1\setminus \bbP^1(k),$$
the affine line and so-called Drinfeld's upper half plane, respectively. 

\vskip5pt

We discuss an interesting example in the case $Y=\bbA^1_k$. For this, we assume that $K$ contains the $p$-th roots of unity $\mu_p$ and we choose an element $\pi\in\fro$ with ${\ord}_p (\pi) = 1/(p-1)$. We have the affine coordinate $t$ on $\bbA^1_{\fro}$ and we let $\partial=d/dt$.
We let ${\mathscr L}_{\varpi}$ be the coherent $\sD^{\dagger}_{\cP}$-module defined by the {\it Dwork overconvergent $F$-isocrystal} $L_{\pi}$ on $\bbY$ associated with $\pi$, i.e.  ${\mathscr L}_{\pi}=v_{!+}L_{\pi}$ where $v: \bbY\rightarrow\bbP$. Recall that the underlying $\cO_{\cP,\bbQ}$-module of ${\mathscr L}_{\pi}$ is $\cO_{\cP,\bbQ}(\infty)$, endowed with a compatible $\sD^{\dagger}_{\cP}$-module structure for which $\partial (1) = - \pi$,
 \cite[4.5.5]{BerthelotIntro}.
 
 \vskip5pt
 
 Write $ \frn = K.e$ with $e=\bigl( \begin{smallmatrix}
  0&1\\ 0&0
\end{smallmatrix} \bigr)$. Let $\eta: \frn \rightarrow K$ be a nonzero character and consider Kostant's {\it standard Whittaker module} 
 
 $$W_{\theta_0, \eta}:= U(\frg) \otimes_{ Z(\frg) \otimes U(\frn) } K_{\theta_0,\eta} $$

with character $\eta$ and central character $\theta_0$ 
\cite[(3.6.1)]{KostantWhittaker}. It is an irreducible $U(\frg)$-module \cite[Thm. 3.6.1]{KostantWhittaker}, but does not lie in $\cO_0$. 
In fact, the restriction of the $\sD_{\bbP_{K}^1}$-module ${\rm Loc}(W_{\theta_0,\eta})$ to $\bbA^1_K$ has an irregular singularity 
at $\infty$ \cite[4.4]{MilicicSoergelWhittaker}.

\vskip5pt
 
Let  $W^{\dagger}_{\theta_0,\eta}:=D^{\dagger}(\GG)\otimes_{U(\frg) } W_{\theta_0,\eta}$.

\begin{thm} \label{Dwork} Consider the character $\eta$ defined by $\eta(e):=\pi$. 
There is a canonical isomorphism $${\mathscr Loc} (W^{\dagger}_{\theta_0,\eta} ) \car {\mathscr L}_{\pi} $$
as left $\sD^{\dagger}_{\cP}$-modules. In particular, $W^{\dagger}_{\theta_0,\eta}$ is geometrically $F$-holonomic.
\end{thm}
\begin{proof}
The standard Whittaker module $W_{\theta_0, \eta}$ admits the presentation 
$$W_{\theta_0, \eta}= U(\frg)/ U(\frg) (e - \eta(e))= U(\frg)/ U(\frg) (e - \pi)$$ whence
$W^{\dagger}_{\theta_0,\eta}= D^{\dagger}(\GG) / D^{\dagger}(\GG) (e-\pi).$ The canonical morphism $U(\frg)\rightarrow \sD_{\bbP_{K}^1}$ maps $e$ to $-\partial$, cf. \cite[11.2.1]{Hotta}, and the 
isomorphism of part (b) in theorem \ref{localizationthm} is compatible with this morphism. We obtain

$${\mathscr Loc} (W^{\dagger}_{\theta_0,\eta}) =\sD^{\dagger}_{\cP}  / \sD^{\dagger}_{\cP} (\partial+\pi)$$

which coincides with the standard presentation of the $\sD^{\dagger}_{\cP}$-module ${\mathscr L}_{\pi}$ 
\cite[Prop. 5.2.3]{Berthelot_Trento}.
\end{proof} 

Remark: It is interesting to note that the Dwork isocrystal ${\mathscr L}_{\pi}$ is {\it algebraic} in the sense that it comes from 
an algebraic $\sD_{\bbP^1_K}$-module, namely ${\rm Loc}(W_{\theta_0,\eta})$, by extension of scalars 
$\overline{\sD}_{\bbP^1_K}\rightarrow \sD^{\dagger}_{\cP}$. 

\vskip5pt 

We discuss an example in the second case, where $Y= \bbP_k^1 \setminus \bbP^1(k)$. We identify $k=\bbF_q$.
We assume that $K$ contains the cyclic group $\mu_{q+1}$ of $(q+1)$-th roots of unity.
We consider the so-called {\it Drinfeld curve}

$$Y'=\Big\{(x,y) \in \bbA_k^2 \midc xy^q-x^qy=1\Big\}.$$

It is an affine smooth irreducible curve and the map  $(x,y)\mapsto [x:y]$ is an unramified Galois covering 

 $$u: Y' \longrightarrow Y$$
 
 with Galois group $\mu_{q+1}$. The group $\mu_{q+1}$ acts by homotheties $\zeta. (x,y)=(\zeta.x,\zeta.y)$. 
 We have a smooth projective compactification 

$$\overline{Y'}=\Big\{[x:y:z] \in \bbP_k^2 \midc xy^q-x^qy=z^{q+1}\Big\}$$ and the covering extends to a smooth (and tamely ramified) morphism 

$$u: \overline{Y'} \longrightarrow \bbP_k^1, $$

given by $[x:y:z]\mapsto [x:y].$ The boundary $Z'=\overline{Y'}\setminus Y'$ is mapped bijectively to $Z=\bbP^1(k)$ and the ramification index at each point in $Z$ is $q+1$. For more details the reader may consult \cite[chap. 2]{Bonnafe}. We denote by $u: \bbY'\rightarrow\bbY$ the morphism of couples induced by $u$.  We let $\cE=\bbR^{*}u_{\rm rig,*}\cO_{\bbY'}$ be the relative rigid cohomology sheaf which, in our situation, is just the direct image of $\cO_{\bbY}$ 
under the morphism $u$ endowed with the Gauss-Manin connection. 

\begin{prop} The relative rigid cohomology sheaf, as an overconvergent $F$-isocrystal on $\bbY$,
admits a decomposition $\cE=\oplus_{j=0,...,q} \cE(j)$, where $\cE(j)$ is the isotypic subspace (of rank one) on which 
$\mu_{q+1}$ acts by the character $\zeta\mapsto\zeta^{j}$. In particular, each pair $(Y,\cE(j))$ corresponds to an irreducible geometrically $F$-holonomic $D^{\dagger}(\GG)_{\theta_0}$-module $H^0(\cP, v_{!+}\cE(j))$.
\end{prop}
\begin{proof}
The cover $u: Y' \rightarrow Y$ is an abelian prime-to-$p$ Galois covering as considered in \cite{GK_DL}. 
The relative rigid cohomology, as an overconvergent $F$-isocrystal on the base $Y$ (denoted there by $E^{\dagger}$) together with its decomposition 
$E^{\dagger}=\oplus_j E^{\dagger}(j)$ is constructed in \cite[sec. 2]{GK_DL}. Note that $u: Y' \rightarrow Y$ is even equal to (one of the $q-1$ connected components of) the Deligne-Lusztig torsor for the nonsplit torus $\mu_{q+1}$ in the finite group $G(\bbF_q)$, a special situation considered in \cite[sec. 4]{GK_DL}.
\end{proof} 
 
Are the modules $H^0(\cP, v_{!+}\cE(j))$ {\it algebraic} in the sense that they arise from irreducible $U(\frg)$-modules, by extension of scalars 
$U(\frg)\rightarrow D^{\dagger}(\GG)$? Let us remark that the {\it th\'eor\`eme d'alg\'ebrisation} of Christol-Mebkhout \cite[thm. 5.0-10]{CtMeIV} implies that any overconvergent $F$-isocrystal on the open $Y$ is algebraic, i.e. comes from an algebraic connection on a characteristic zero lift of $Y$.
However, this does not imply (at least a priori) that the intermediate extensions preserve this algebraicity. To our knowledge, the most general result in this direction at the moment is our theorem \ref{mainappli} above. 
\vskip5pt 
If the modules $H^0(\cP, v_{!+}\cE(j))$ are algebraic, to which class do they belong? We recall that irreducible $U(\frg)$-modules fall into three classes: highest weight modules, Whittaker modules and a third class whose objects (with a fixed central character) are in bijective correspondence with similarity classes of irreducible elements of a certain localization of the first
Weyl algebra \cite{Block}. We plan to come back to these question in future work. 

\vskip5pt

We finish this paper with the remark, still in the case (2), that if we concentrate on the subcategory of 
overconvergent $F$-isocrystals on $Y=\bbP_1^k \setminus Z$ 
which are {\it unit-root}, then work of Tsuzuki \cite[Thm. 7.2.3]{TsuzukiFinite} shows that 
this category is equivalent to the category of $p$-adic representations of the \'etale fundamental group $\pi^{et}_1(Y)$ with finite monodromy
(i.e. representations such that for each  $y\in Z$ the inertia subgroup at $y$ acts through a finite quotient). Of course, 
the trivial representation corresponds to the constant isocrystal $\cO_{\bbY}$. 

\bibliographystyle{plain}
\bibliography{mybib}

\end{document}